\documentclass[a4paper,11pt]{amsart}
\usepackage{amsmath,amsthm,amsfonts,amssymb,mathrsfs}
\usepackage{enumerate}
\usepackage[colorlinks=true, linkcolor=blue, citecolor=magenta, menucolor=black]{hyperref}
\usepackage{geometry}
\usepackage[demo]{graphicx}
\usepackage[up]{caption}
\usepackage{subcaption}
\usepackage{pgf,tikz}
\usepackage[toc,page]{appendix}
\usetikzlibrary{arrows}
\usetikzlibrary{patterns}
\usepackage{color}

\numberwithin{equation}{section}
\theoremstyle{plain}
\newtheorem{thm}{Theorem}[section]

\newtheorem{lem}[thm]{Lemma}
\newtheorem{prop}[thm]{Proposition}
\newtheorem{cor}[thm]{Corollary}

\newtheorem{letterthm}{Theorem}

\newtheorem{lettercor}[letterthm]{Corollary}

\theoremstyle{definition}
\newtheorem{defn}[thm]{Definition}
\newtheorem*{defn*}{Definition}

\newtheorem{exmp}[thm]{Example}

\newtheorem{rem}[thm]{Remark}

\newtheorem*{notation}{Notation}

\newtheorem{claim}[thm]{Claim}

\newcommand{\N}{\mathbf{N}}
\newcommand{\R}{\mathbf{R}}
\newcommand{\C}{\mathbf{C}}
\newcommand{\Q}{\mathbf{Q}}
\newcommand{\Z}{\mathbf{Z}}

\newcommand{\cA}{\mathcal{A}}
\newcommand{\cB}{\mathcal{B}}
\newcommand{\cC}{\mathcal{C}}

\newcommand{\cK}{\mathcal{K}}
\newcommand{\cM}{\mathcal{M}}
\newcommand{\cN}{\mathcal{N}}

\newcommand{\cP}{\mathcal{P}}
\newcommand{\cQ}{\mathcal{Q}}

\newcommand{\cS}{\mathcal{S}}

\newcommand{\cU}{\mathcal{U}}

\newcommand{\cZ}{\mathcal{Z}}

\newcommand{\tpi}{{\widetilde{\pi}}}

\newcommand{\tsigma}{{\widetilde{\sigma}}}

\newcommand{\oV}{{\overline V}}
\newcommand{\oU}{{\overline U}}
\newcommand{\ov}{{\overline v}}
\newcommand{\ou}{{\overline u}}

\newcommand{\Stab}{\operatorname{Stab}}
\newcommand{\Fix}{\operatorname{Fix}}

\newcommand{\ot}{\otimes}
\newcommand{\ovt}{\mathbin{\overline{\otimes}}}
\newcommand{\Aut}{\operatorname{Aut}}
\newcommand{\Ad}{\operatorname{Ad}}

\newcommand{\Ind}{\operatorname{Ind}}
\newcommand{\id}{\operatorname{id}}

\newcommand{\SL}{\operatorname{SL}}
\newcommand{\EL}{\operatorname{EL}}
\newcommand{\SO}{\operatorname{SO}}

\newcommand{\Prob}{\operatorname{Prob}}
\newcommand{\supp}{\operatorname{supp}}
\newcommand{\Sub}{\operatorname{Sub}}

\newcommand{\bary}{\operatorname{Bar}}
\newcommand{\dpr}{^{\prime\prime}}

\newcommand{\eps}{\varepsilon}
\newcommand{\actson}{\curvearrowright}

\newcommand{\rE}{\operatorname{ E}}
\newcommand{\rC}{\operatorname{C}}
\newcommand{\rL}{\operatorname{ L}}
\newcommand{\Int}{\operatorname{Int}}
\newcommand{\Ball}{\operatorname{Ball}}
\newcommand{\Har}{\operatorname{Har}}
\newcommand{\Rep}{\operatorname{Rep}}
\newcommand{\Aff}{\operatorname{Aff}}
\newcommand{\PSL}{\operatorname{PSL}}
\newcommand{\dd}{{\,\mathrm d}}

\allowdisplaybreaks

\title[Stationary characters on lattices of semisimple Lie groups]{Stationary characters on lattices \\ of semisimple Lie groups}

\author{R\'emi Boutonnet}
\address{Institut de Math\'ematiques de Bordeaux \\ CNRS \\ Universit\'e Bordeaux I \\ 33405 Talence \\ FRANCE}
\email{remi.boutonnet@math.u-bordeaux.fr}
\thanks{RB is supported by a PEPS grant from CNRS and ANR grant AODynG 19-CE40-0008}

\author{Cyril Houdayer}
\address{Universit\'e Paris-Saclay \\ Institut Universitaire de France \\  Laboratoire de Math\'ematiques d'Orsay\\ CNRS \\ 91405 Orsay\\ FRANCE}
\email{cyril.houdayer@universite-paris-saclay.fr}
\thanks{CH is supported by ERC Starting Grant GAN 637601 and Institut Universitaire de France}


\begin{document}

\begin{abstract}
We show that stationary characters on irreducible lattices $\Gamma < G$ of higher-rank connected semisimple Lie groups are conjugation invariant, that is, they are genuine characters. This result has several applications in representation theory, operator algebras, ergodic theory and topological dynamics. In particular, we show that for any such irreducible lattice $\Gamma < G$,  the left regular representation $\lambda_\Gamma$ is weakly contained in any weakly mixing representation $\pi$. We prove that for any such irreducible lattice $\Gamma < G$, any Uniformly Recurrent Subgroup (URS) of $\Gamma$ is finite, answering a question of Glasner--Weiss. We also obtain a new proof of Peterson's character rigidity result for irreducible lattices $\Gamma < G$. The main novelty of our paper is a structure theorem for stationary actions of lattices on von Neumann algebras. 
\end{abstract}

\maketitle

\section{Introduction and statement of the main results}

A major achievement in the theory of discrete subgroups of semisimple Lie groups is Margulis' {\em superrigidity theorem}. The statement is as follows: whenever $G$ is a connected semisimple Lie group with trivial center, no compact factor and real rank at least two, $\Gamma < G$ is an irreducible lattice and $H$ is a simple Lie group with trivial center, any homomorphism $\pi : \Gamma \to H$ such that $\pi(\Gamma)$ is Zariski dense in $H$ and not relatively compact in $H$ extends to a continuous homomorphism $\pi : G \to H$ (see \cite[Chapter VII]{Ma91} for more general statements). Connes suggested that there should be a rich analogy between the embedding of a lattice in its ambient Lie group and the embedding of a lattice in its ambient group von Neumann algebra (see \cite{Jo00}). Notably, an operator algebraic version of Margulis' superrigidity theorem was expected to hold, and this was confirmed by recent advances. Interestingly, these recent advances build on another famous result of Margulis; his normal subgroup theorem \cite[Theorem IV.4.10]{Ma91}), which is not directly related to his superrigidity statement.

The first {\em operator algebraic superrigidity} theorem was obtained by Bekka \cite{Be06} who showed that whenever $\Gamma = \PSL_n(\Z)$ with $n \geq 3$ and $M$ is a type ${\rm II_1}$ factor, any homomorphism $\pi : \Gamma \to \mathcal U(M)$ such that $\pi(\Gamma)\dpr = M$ extends to a normal unital $\ast$-isomorphism $\pi : \rL(\Gamma) \to M$. Recently, Peterson \cite{Pe14} obtained a far-reaching generalization of Bekka's result by showing that any irreducible lattice $\Gamma < G$ in a property (T) connected semisimple Lie group with trivial center and real rank at least two is operator algebraic superrigid in the above sense (see also \cite{CP13}). Operator algebraic superrigidity can be reformulated as a classification problem for characters on the group. 

Recall that a {\it character} on a countable discrete group $\Lambda$ is a positive definite function which is invariant under conjugation, and which is normalized to map the identity element to $1$. Thanks to the GNS construction, any character on $\Lambda$ gives rise to a homomorphism of $\Lambda$ into the unitary group of a tracial von Neumann algebra. For example, the natural embedding of $\Lambda$ in its group von Neumann algebra $\rL(\Lambda)$ corresponds to the Dirac character $\delta_e$ at the trivial element $e$. A rich source of characters comes from ergodic theory of group actions. Indeed, for any probability measure preserving (pmp) action $\Lambda \curvearrowright (X, \nu)$, the map $\varphi : \Lambda \to \C : \gamma \mapsto \nu(\{x \in X \mid \gamma x = x\})$ defines a character on $\Lambda$. Then $\varphi = \delta_e$ if and only if the action $\Lambda \curvearrowright (X, \nu)$ is essentially free. The set of characters on a countable discrete group $\Lambda$ is a convex set which is compact with respect to pointwise convergence. The group $\Lambda$ is operator algebraic superrigid in the above sense if and only if every extreme point $\varphi$ in the space of characters of $\Lambda$ is either almost periodic (i.e.\ the corresponding GNS representation is finite dimensional) or $\varphi = \delta_e$. Note that such a classification of characters for lattices strengthens both Margulis'\ normal subgroup theorem (see \cite[Theorem IV.4.10]{Ma91}) and Stuck--Zimmer's rigidity result on stabilizers of pmp ergodic actions (see \cite[Corollary 4.4]{SZ92}). See \cite{Be19, BF20, DM12, LL20, PT13} for other character rigidity results. 

A generic construction for characters goes as follows. If $\pi: \Lambda \to \cU(A)$ is a group homomorphism into the set of unitary elements of a unital $\rC^*$-algebra $A$, and if $\tau$ is a tracial state on $A$, then $\tau \circ \pi$ is a character on $\Lambda$. But this construction requires that $A$ admits a trace, which is not always the case. This situation is analogous to the commutative setting, where $\Lambda$ acts continuously on a compact space $X$: there is not always a $\Lambda$-invariant Borel probability measure on the space $X$. In this respect, Furstenberg introduced the notion of {\em stationary measure} and stationary action \cite{Fu62a, Fu62b}. The key point is that, for any continuous action $\Lambda \curvearrowright X$ on a compact space, there always exists a stationary Borel probability measure on $X$. For this reason, the concept of stationary action plays an important role in the study of nonamenable groups and in Furstenberg boundary theory.  Note that Furstenberg boundary theory was one of the key tools Margulis used in the proof of his superrigidity theorem (\cite{Ma91}). 

This concept of stationarity was recently used in the noncommutative setting by Hartman--Kalantar \cite{HK17} in the context of $\rC^*$-simplicity of groups. They investigate a stationary version of characters, which will also be our main object of study. 

\begin{notation} Before stating our main results, we introduce the following notation that we will use throughout the introduction.
\begin{itemize}
\item Let $G$ be any connected semisimple Lie group with finite center and no nontrivial compact factor, all of whose simple factors have reak rank at least two. Choose a maximal compact subgroup $K < G$ and a minimal parabolic subgroup $P<G$, so that $G = KP$.  For instance, for every $n \geq 3$, let $G = \SL_n(\R)$ and choose $K = \SO_n(\R)$ and $P < G$ the subgroup of upper triangular matrices.
\item We denote by $\nu_P \in \Prob(G/P)$ the unique $K$-invariant Borel probability measure on the homogeneous space $G/P$. More generally, if $P \subset Q \subset G$ is a parabolic subgroup, we denote by $\nu_Q \in \Prob(G/Q)$ the unique $K$-invariant Borel probability measure on the homogeneous space $G/Q$. Observe that for every parabolic subgroup $P \subset Q \subset G$, the probability measure $\nu_Q \in \Prob(G/Q)$ is $G$-quasi-invariant.
\end{itemize}
\end{notation}

The following concept will be central in our paper.
\begin{defn*}
Let $G$ be as in the notation. Let $\Gamma< G$ be any lattice. We say that a probability measure $\mu_0 \in \Prob(\Gamma)$ is {\em Furstenberg} if the following three conditions are satisfied:
\begin{itemize}
\item [$(\rm i)$] The support of $\mu_0$ is equal to $\Gamma$;
\item [$(\rm ii)$]  $\mu_0 \ast \nu_P = \nu_P$, that is, $\nu_P$ is $\mu_0$-stationary;
\item [$(\rm iii)$] The space $(G/P, \nu_P)$ is the Poisson boundary associated with the simple random walk on $\Gamma$ with law $\mu_0$ (see \cite{Fu62b, Fu00}).
\end{itemize}
\end{defn*}

By a result of Furstenberg \cite[Theorem 3]{Fu67} (see also \cite[Theorem 2.21]{Fu00}), there always exists a Furstenberg probability measure $\mu_0 \in \Prob(\Gamma)$. Moreover, for every parabolic subgroup $P \subset Q \subset G$, $\nu_Q$ is the unique $\mu_0$-stationary measure on the homogeneous space $G/Q$ (see \cite{Fu73, GM89}).

Denote by $\mathcal P(\Gamma)$ the weak$^*$-compact convex space of positive definite functions on $\Gamma$. We say that $\varphi \in \mathcal P(\Gamma)$ is {\em normalized} if $\varphi(e) = 1$. Let $\mu \in \Prob(\Gamma)$ be any probability measure. We say that a normalized positive definite function $\varphi \in \mathcal P(\Gamma)$ is a $\mu$-{\em character} if $\sum_{\gamma \in \Gamma} \mu(\gamma) \, \varphi(\gamma^{-1}g\gamma) = \varphi(g)$, for all $g \in \Gamma$. Any character on $\Gamma$ is obviously a $\mu$-character. Conversely, our first main result shows that any $\mu_0$-character is a genuine character.

\begin{letterthm}\label{main stationary characters}
Let $G$ be as in the notation. Let $\Gamma < G$ be any irreducible lattice and $\mu_0 \in \Prob(\Gamma)$ any Furstenberg probability measure. Then any $\mu_0$-character $\varphi$ on $\Gamma$ is conjugation invariant, that is, $\varphi$ is a genuine character.
\end{letterthm}

Before describing the applications of Theorem \ref{main stationary characters}, let us explain the main technical novelty used in the proof, which is of independent interest. Let $M$ be any von Neumann algebra, $\phi \in M_\ast$ any normal state and $\sigma : \Gamma \curvearrowright M$ any action. We simply write $\gamma \phi = \phi \circ \sigma_\gamma^{-1} \in M_\ast$ for every $\gamma \in \Gamma$. We say that the action $\sigma : \Gamma \curvearrowright M$ is {\em ergodic} if the fixed-point subalgebra $M^\Gamma = \{x \in M \mid \sigma_\gamma(x) = x, \forall \gamma \in \Gamma\}$ satisfies  $M^\Gamma = \C 1$. We say that the state $\phi \in M_\ast$ is $\mu_0$-{\em stationary} if $\sum_{\gamma \in \Gamma} \mu_0(\gamma) \, \gamma \phi = \phi$. If the action $\sigma : \Gamma \curvearrowright M$ is ergodic and the state $\phi \in M_\ast$ is $\mu_0$-stationary, we say that $(M, \phi)$ is an ergodic $(\Gamma, \mu_0)$-von Neumann algebra. In the case when $(M, \phi) = \rL^\infty(X, \nu)$ for some standard probability space $(X, \nu)$ where $\phi$ is given by integration against $\nu$, we say that $(X, \nu)$ is an ergodic $(\Gamma, \mu_0)$-space.

The following result about $(\Gamma,\mu_0)$-von Neumann algebras is in the spirit of Nevo--Zimmer's work \cite{NZ97,NZ00}, see also the survey \cite{NZ02}.

\begin{letterthm}\label{main NCNZ}
Let $G$ be as in the notation. Let $\Gamma < G$ be any lattice and $\mu_0 \in \Prob(\Gamma)$ any Furstenberg probability measure. Let $(M, \phi)$ by any ergodic $(\Gamma, \mu_0)$-von Neumann algebra. Then the following dichotomy holds. 
\begin{itemize}
\item Either $\phi$ is $\Gamma$-invariant.
\item Or there exist a proper parabolic subgroup $P \subset Q \subsetneq G$ and a $\Gamma$-equivariant normal unital $\ast$-embedding $\theta : \rL^\infty(G/Q, \nu_Q) \to M$ such that $\phi \circ \theta = \nu_Q$.
\end{itemize}
\end{letterthm}

Theorem \ref{main NCNZ} is reminiscent of \cite[Theorem 1]{NZ00}, but two differences appear: 
\begin{itemize}
\item Our theorem is about stationary actions of the lattice $\Gamma$, while Nevo--Zimmer's result is about stationary actions of the ambient Lie group $G$.
\item We deal with general von Neumann algebras, while Nevo--Zimmer's theorem only deals with measure spaces, i.e.\ with {\em commutative} von Neumann algebras.
\end{itemize}

Firstly, we reduce to a statement for $G$-actions by using induction and disintegration. This is based on a simple, surprisingly unnoticed observation. For any ergodic $(\Gamma, \mu_0)$-von Neumann algebra $(M, \phi)$, we construct a $\mu$-stationary normal state $\varphi$ on the induced von Neumann algebra $\Ind_\Gamma^G (M)$, where $\mu \in  \Prob(G)$ is a $K$-invariant admissible Borel probability measure (see Theorem \ref{induced stationary state}). This is where we use that the probability measure $\mu_0 \in \Prob(\Gamma)$ is {\em Furstenberg}. We refer to Section \ref{section:induced} for further details. This simple observation is new even in the commutative setting, and easily combines with \cite[Theorem 1]{NZ00}, to yield Theorem \ref{letterthm:actions} below.

Secondly, we prove the analogous statement of Theorem \ref{main NCNZ} but for $G$-actions instead (see Theorem \ref{thm:NZ}). This generalization of Nevo--Zimmer theorem to the noncommutative setting presents both technical and conceptual difficulties and is not a mere adaptation of their original proof. The proof of Theorem \ref{thm:NZ} nevertheless relies on Nevo--Zimmer's result \cite[Theorem 1]{NZ00}. We refer to Section \ref{section:NZ} for a more detailed explanation of the global strategy of the proof.

Let us now explain the proof of Theorem \ref{main stationary characters}. Starting from an extreme point $\varphi$ in the weak$^*$-compact convex set of $\mu_0$-characters, denote by $\pi_\varphi$ its GNS representation and regard $M = \pi_\varphi(\Gamma)\dpr$ as a $(\Gamma, \mu_0)$-von Neumann algebra where the action $\Gamma \curvearrowright \pi_\varphi(\Gamma)\dpr$ is given by conjugation. Applying Theorem \ref{main NCNZ} to $M$ and exploiting techniques that recently appeared in the characterization of $\rC^*$-simplicity via Furstenberg boundary (see \cite{KK14, BKKO14, Ha15, HK17}), we show that $\varphi$ is necessarily conjugation invariant.

Surprisingly, Theorem \ref{main NCNZ} also allows to classify {\em genuine} characters on irreducible lattices $\Gamma < G$. More precisely, we obtain a new proof of Peterson's {\em character rigidity} result  \cite{Pe14}. 

\begin{letterthm}[Peterson, \cite{Pe14}]\label{main peterson}
Let $G$ be as in the notation and assume moreover that $G$ has trivial center. Let $\Gamma < G$ be any irreducible lattice. Then any extreme point $\varphi$ in the space of characters of $\Gamma$ is either almost periodic or $\varphi = \delta_e$.
\end{letterthm}

We point out that our approach requires that all simple factors of $G$ have real rank at least two (as in the notation), while Peterson's character rigidity result holds more generally when $G$ is a property (T) connected semisimple Lie group with trivial center, no nontrivial compact factor and real rank at least two. 

Let us now turn our attention to applications of Theorems \ref{main stationary characters} and \ref{main peterson} to representation theory and $\rC^*$-algebras. By \cite{BCH94}, when $G$ is a noncompact semisimple Lie group with trivial center, any irreducible lattice $\Gamma< G$ is $\rC^*$-simple and has the unique trace property, that is, the reduced $\rC^*$-algebra $\rC^*_\lambda(\Gamma)$ is simple and has a unique tracial state $\tau_\Gamma$ (see also \cite{BKKO14} for a new approach). The next corollary provides a far-reaching generalization of this phenomenon to {\em arbitrary} weakly mixing representations. Recall that a unitary representation $\pi$ is called weakly mixing if $\pi$ does not contain any nonzero finite dimensional subrepresentation.

\begin{lettercor}\label{main rep}
Let $G$ be as in the notation and assume moreover that $G$ has trivial center. Let $\Gamma < G$ be any irreducible lattice. Then for any weakly mixing representation $\pi : \Gamma \to \mathcal U(H_\pi)$, the left regular representation $\lambda_\Gamma$ is weakly contained in $\pi$. Moreover, if we denote by $\Theta_{\pi, \lambda} : \rC^*_\pi(\Gamma) \to \rC^*_\lambda(\Gamma) : \pi(\gamma) \mapsto \lambda_\Gamma(\gamma)$ the corresponding surjective unital $\ast$-homomorphism, then 
\begin{itemize}
\item [$(\rm i)$] $\tau_\Gamma \circ \Theta_{\pi, \lambda}$ is the unique tracial state on $\rC_\pi^*(\Gamma)$. 
\item [$(\rm ii)$] $\ker(\Theta_{\pi, \lambda})$ is the unique proper maximal ideal of $\rC^*_\pi(\Gamma)$.
\end{itemize}
\end{lettercor}

The proof of Corollary \ref{main rep} relies on the following simple observation that was recently exploited  by Hartman--Kalantar \cite{HK17}. Starting from a unitary representation $\pi : \Gamma \to \mathcal U(H_\pi)$, the conjugation action $\Gamma \curvearrowright \rC^*_\pi(\Gamma)$ possesses a $\mu_0$-stationary state $\phi$. We can thus apply Theorem \ref{main stationary characters} to deduce that the $\mu_0$-character $\phi \circ \pi$ is actually a character. Then Theorem \ref{main peterson} allows to conclude.

Corollary \ref{main rep} was previously known only for particular examples of weakly mixing unitary representations (see e.g.\ \cite{Be95,BK19}). Let us point out that any countable group $\Lambda$ satisfying the conclusion of Corollary \ref{main rep} is {\em just infinite} in the sense that any nontrivial normal subgroup has finite index; moreover, any faithful properly ergodic pmp action $\Lambda \curvearrowright (X, \nu)$ is essentially free; also, $\Lambda$ is {\em character rigid} in the sense of Theorem \ref{main peterson}. Thus, one might also regard Corollary \ref{main rep} as a strengthening of Margulis' normal subgroup theorem (\cite[Chapter IV]{Ma91}), Stuck--Zimmer's rigidity result \cite{SZ92} and Peterson's character rigidity result \cite{Pe14}. 

Let us now turn our attention to applications of Theorems \ref{main stationary characters} and \ref{main peterson} to ergodic theory. As a straightforward consequence of Theorem \ref{main NCNZ} and Stuck--Zimmer's result \cite[Corollary 4.4]{SZ92}, we obtain the following structure theorem for stationary actions of irreducible lattices.

\begin{letterthm}\label{letterthm:actions}
Let $G$ be as in the notation. Let $\Gamma < G$ be any irreducible lattice and $\mu_0 \in \Prob(\Gamma)$ any Furstenberg probability measure. Let $(X, \nu)$ be any ergodic $(\Gamma, \mu_0)$-space. Then the following dichotomy holds.
\begin{itemize}
\item Either $\nu$ is $\Gamma$-invariant.
\item Or there exist a proper parabolic subgroup $P \subset Q\subsetneq G$ and a $\Gamma$-equivariant measurable factor map $(X, \nu) \to (G/Q, \nu_Q)$.
\end{itemize}
Moreover, if the action $\Gamma \curvearrowright (X, \nu)$ is faithful and properly ergodic, then it is essentially free.
\end{letterthm}

Finally, we apply Theorem \ref{letterthm:actions} to topological dynamics. Denote by $\Sub(\Gamma)$ the compact metrizable space of all subgroups of $\Gamma$ endowed with the Chabauty topology. Define the conjugation action $\Gamma \curvearrowright \Sub(\Gamma)$ by $\gamma \cdot \Lambda = \gamma \Lambda \gamma^{-1}$ for every $\gamma \in \Gamma$ and every $\Lambda \in \Sub(\Gamma)$. In the measurable setting, following \cite{AGV12}, an {\em invariant random subgroup} (IRS for short) is a conjugation invariant Borel probability measure $\nu \in \Prob(\Sub(\Gamma))$. By Stuck--Zimmer's result \cite[Corollary 4.4]{SZ92}, any ergodic IRS of an irreducible lattice $\Gamma < G$, where $G$ is as in the notation and with trivial center, is finite. In the topological setting, following \cite{GW14}, a {\em Uniformly Recurrent Subgroup} (URS for short) is a closed minimal $\Gamma$-invariant subset $X \subset \Sub(\Gamma)$. The study of URSs has received a lot of attention recently due to its connections with IRSs (\cite{7s12, Ge14}) and $\rC^*$-simplicity (\cite{Ke15, LBMB16}). 

Recall that an action by homeomorphisms $\Lambda \curvearrowright X$ of a countable discrete group $\Lambda$ on a compact metrizable space $X$ is {\em topologically free} if for any $\gamma \neq e$, the closed subset $\Fix(\gamma) = \{x \in X \mid \gamma x = x\}$ of $\gamma$-fixed points has empty interior in $X$. The next corollary provides a topological analogue of Stuck--Zimmer's rigidity result \cite{SZ92} and answers positively a question raised by Glasner--Weiss (see \cite[Problem 5.4]{GW14}).

\begin{lettercor}\label{lettercor:URS}
Let $G$ be as in the notation and assume moreover that $G$ has trivial center. Let $\Gamma < G$ be any irreducible lattice. For any minimal action $\Gamma \curvearrowright X$ on a compact metrizable space, either $X$ is finite or the action is topologically free. In particular, any {\em URS} of $\Gamma$ is finite.
\end{lettercor}

\subsection*{Acknowledgments} It is our pleasure to thank Uri Bader, Bachir Bekka and Yves Benoist for useful discussions, and Narutaka Ozawa and Stefaan Vaes for their valuable remarks. We would like also to thank the organizers of the conference {\em Dynamics of group actions} held in Cetraro (Italy) in May 2019 as well as the organizers of the Workshop $\rC^*$-{\em algebras} held in Oberwolfach (Germany) in August 2019, where parts of this work were done. Last but not least, we are greatful to the two referees for their careful reading and for providing insightful remarks that helped improve the exposition of the paper.

{
  \hypersetup{linkcolor=black}
  \tableofcontents
}

\section{Preliminaries}

\subsection{Group actions on von Neumann algebras}\label{subsection:vN}

This paper appeals to the theory of von Neumann algebras and C*-algebras. While we will not recall the very basic definitions of the theory, we will try to keep the pre-requisite at a minimum.

In this section, $G$ will be any locally compact second countable (lcsc) group and $\mu \in \Prob(G)$ any {\em admissible} Borel probability measure. This means that the support of $\mu$ generates $G$ as a semigroup and contains the trivial element in its interior, and that $\mu$ is absolutely continuous with respect to the Haar measure on $G$.

Let $A$ be any separable unital $\rC^*$-algebra. We denote by $\cS(A)$ the set of all states on $A$, i.e.\ of all linear functionals $\phi: A \to \C$ such that $\phi(x^*x) \geq 0$ for all $x \in A$ ({\em positive}), and $\phi(1) = 1$ ({\em normalized}). It is a compact convex subset of $A^*$ for the weak$^*$-topology.

We endow $A$ with its norm topology. We say that an action $\sigma : G \curvearrowright A$ by $\ast$-automorphisms is continuous (or is a {\em $\rC^*$-action}) if the map $G \times A \to A : (g, a) \mapsto \sigma_g(a)$ is continuous. The action $G \curvearrowright A$ induces a continuous affine action $G \curvearrowright \mathcal S(A)$, called the push-forward action, defined by the formula 
$$\forall g \in G, \forall\psi \in \mathcal S(A), \forall a \in A, \quad g \psi (a) = \psi(\sigma_g^{-1}(a)).$$
We say that a state $\psi \in \mathcal S(A)$ is $\mu$-{\em stationary} if for every $a \in A$, we have
$$\psi(a) = \int_G g \psi(a) \, {\rm d}\mu(g).$$

Let $M$ be any von Neumann algebra with separable predual $M_\ast$. We endow $M$ with the ultraweak (i.e.\ weak$^*$) topology coming from the canonical identification $M = (M_\ast)^*$.  In this manner, $M_*$ is identified with the subset of $M^*$ consisting of all ultraweakly continuous linear functionals, also called {\em normal} linear functionals. We say that an action $\sigma : G \curvearrowright M$ by $\ast$-automorphisms is continuous (or is a von Neumann action) if the map $G \times M \to M : (g, x) \mapsto \sigma_g(x)$ is continuous. By \cite[Proposition X.1.2]{Ta03a}, this definition is equivalent to saying that for every  $x \in M$ and every $\phi \in M_\ast$, the map $G \to \C : g \mapsto \phi(\sigma_g(x))$ is continuous. In this case, the action $G \curvearrowright M$ induces an affine norm continuous action $G \curvearrowright M_\ast$ defined by the formula 
\[\forall g \in G, \forall\phi \in M_\ast, \forall x \in M, \quad g \phi (x) = \phi(\sigma_g^{-1}(x)).\]
We say that a normal state $\phi \in M_\ast$ is $\mu$-{\em stationary} if for every $x \in M$, we have
\[\phi(x) = \int_G g \phi(x) \, {\rm d}\mu(g).\]
In that case, we say that $(M, \phi)$ is a $(G, \mu)$-von Neumann algebra. 

A reader unfamiliar with these notions should keep in mind the commutative examples of the C*-algebra $\mathrm C(X)$ consisting of continuous functions on a compact metrizable space $X$, and the von Neumann algebra $M = \mathrm L^\infty(X,\mu)$ associated to a measure $\mu$ on $X$. In this case, a state on $A$ corresponds to a Borel probability measure on $X$, while a normal state on $M$ corresponds to a probability measure which is absolutely continuous with respect to $\mu$. A C*-action $G \actson A$ corresponds to a continuous action $G \actson X$, while a continuous von Neumann action $G \actson M$ corresponds to a non-singular measurable action $G \actson (X,\mu)$.

The two categories offer different advantages: 
\begin{itemize}
\item In the category of unital $\rC^*$-algebras the weak$^*$-compactness of the state space is very useful (see Example \ref{exmp:C*1} below), whereas the convex space of normal states on a von Neumann algebra does not carry any natural Hausdorff compact topology.
\item On the other hand, the category of von Neumann algebras is more manageable to work with, because of its measurable nature. In this respect, our noncommutative Nevo--Zimmer theorem, Theorem \ref{thm:NZ} below, does not have a $\rC^*$-analogue.
\end{itemize}

It will be useful to switch between these two categories. On the one hand, given a unital C*-algebra $A$, the GNS construction associates to every state $\phi \in \cS(A)$ a unique C*-representation $\pi_\phi: A \to \mathbf B(H)$ and we thus get a von Neumann algebra $M = \pi_\phi(A)\dpr$. Some aspects of this construction are discussed in Section \ref{Section AC}. On the other hand, given a $G$-von Neumann algebra $M$, a regularisation argument can be used to find in $M$ an ultraweakly dense unital $\rC^*$-subalgebra $A$ on which the action is norm continuous, see \cite[Proposition XIII.1.2]{Ta03b}. Since $G$ is assumed to be second countable, if $M$ has separable predual, we may find a separable such $\rC^*$-algebra $A$. In analogy with the commutative setting, $A$ can be called a {\em compact model} of $M$.

Let us further mention specificities of the von Neumann category. A normal state $\phi$ on a von Neumann algebra $M$ is said to be {\em faithful} if $\phi(x^*x) \neq 0$ for every nonzero $x \in M$. This definition also makes sense for any state on a $\rC^*$-algebra, but in the von Neumann algebraic case, faithfulness can be measured by a projection in $M$, called the {\em support} of $\phi$. By definition it is the smallest projection $q \in M$ such that $\phi(q) = 1$. One checks that $\phi$ is faithful on $M$ if and only if $q = 1$.

Given a continuous von Neumann action $G \actson M$, we denote by $M^G = \{x \in M \mid \sigma_g(x) = x, \forall g \in G\}$ the fixed-point von Neumann subalgebra. We say that the action $G \curvearrowright M$ is {\em ergodic} if $M^G = \C 1$. We will need the following useful result. 

\begin{lem}\label{lem:support}
Let $(M, \phi)$ be any $(G, \mu)$-von Neumann algebra. Then the support $q$ of $\phi$ is a $G$-invariant projection. In particular, if the action $G \curvearrowright M$ is ergodic, then $\phi$ is faithful.
\end{lem}
\begin{proof}
Since $\phi \in M_\ast$ is $\mu$-stationary, we have 
$$1 = \phi(q) =  \int_G g \phi(q) \, {\rm d}\mu(g).$$
Since for every $g \in G$, $0 \leq g\phi(q) \leq 1$, it follows that for $\mu$-almost every $g \in G$, we have $g\phi(q) = 1$. Since the map $G \to \C : g \mapsto g\phi(q)$ is continuous, it follows that $g\phi(q) = 1$ for every $g \in \supp(\mu)$. Since $\phi \in M_\ast$ is $\mu^{\ast n}$-stationary for every $n \geq 1$, the same argument shows that $g\phi(q) = 1$ for every $n \geq 1$ and every $g \in \supp(\mu^{\ast n})$. Since $\supp(\mu)^n \subset \supp(\mu^{\ast n})$ for every $n \geq 1$ and since $\bigcup_{n \geq 1} \supp(\mu)^n = G$, it follows that $g\phi(q) = 1$ for every $g \in G$. By definition of $q = \supp(\phi)$, this implies that $q \leq \sigma_g^{-1}(q)$ for every $g \in G$. Applying this inequality at an element $g$ and its inverse, we find $q = \sigma_g(q)$ for every $g \in G$, i.e.\ $q \in M^G$.
\end{proof}

We discuss general facts regarding induction in von Neumann algebraic framework (we refer to \cite[Section X.4]{Ta03a} for further details). Let $G$ be any lcsc group and $H< G$ any closed subgroup. Denote by $\lambda : G \curvearrowright \rL^\infty(G)$ (resp.\ $\rho : G \curvearrowright \rL^\infty(G)$) the {\em left} (resp.\ {\em right}) translation action. We fix a $G$-quasi-invariant Borel probability measure $\nu_{G/H} \in \Prob(G/H)$. Let $N$ be any von Neumann algebra with separable predual and $\sigma : H \curvearrowright N$ any continuous action. Regard $\rL^\infty(G) \ovt N$ as the von Neumann algebra $\rL^\infty(G, N)$ of all essentially bounded measurable functions $f : G \to N$, modulo equality Haar-almost everywhere. We define the {\em induced von Neumann algebra} $\Ind_H^G(N)$ by the formula
$$\Ind_H^G(N) = \left\{f \in \rL^\infty(G) \ovt N \mid \forall h \in H, (\rho_h \otimes \id_N)(f) = (\id_G \otimes \sigma_h^{-1})(f) \right\}.$$ 
We define the continuous {\em induced action} $\Ind(\sigma) : G \curvearrowright \Ind_H^G( N)$  by the formula 
\[ \forall g \in G, \quad \Ind(\sigma)_g = \lambda_g \otimes \id_N.\]
Note that a $G$-invariant element in $\Ind_H^G(N)$ is in particular invariant inside $\rL^\infty(G) \ovt N$, so it belongs to $(\C 1 \ot N) \cap \Ind_H^G(N) = \C 1 \ot N^H$. In particular, the action $\sigma : H \curvearrowright N$ is ergodic if and only if the induced action $\Ind(\sigma) : G \curvearrowright \Ind_H^G(N)$ is ergodic. 

Let us also give a useful alternative description of the induced action. Fix a measurable section $\tau : G/H \to G$ of the natural projection $G \to G/H$ such that $\tau(H) = e$ and denote by $c_\tau : G \times G/H \to H$ the measurable $1$-cocycle associated with $\tau$. By definition, we have $c_\tau(g, w) = \tau(gw)^{-1} g \tau(w)$ for every $g \in G$ and every $w \in G/H$. Regard $\rL^\infty(G/H) \ovt N$ as the von Neumann algebra $\rL^\infty(G/H, N)$ of all essentially bounded measurable functions $f : G/H \to N$, modulo equality $\nu_{G/H}$-almost everywhere. Using $\tau$, we may define a continuous action $\beta^{\tau} : G \curvearrowright \rL^\infty(G/H) \ovt N$ by the formula
$$\forall F \in \rL^\infty(G/H) \ovt N, \forall g \in G, \quad (\beta^{ \tau} )_g(F)(w) = \sigma_{c_\tau(g, g^{-1} w)}(F(g^{-1}w)).$$ 
Using $\tau$, we may also construct a surjective normal unital $\ast$-isomorphism $\pi^{ \tau} : \rL^\infty(G/H) \ovt N \to \Ind_H^G(N)$ by the formula
$$\forall F \in \rL^\infty(G/H) \ovt N, \forall g \in G,  \quad \pi^{ \tau}(F)(g) = \sigma_{c_\tau(g^{-1}, gH)}(F(gH)).$$
Then $\pi^{\tau}$ intertwines the continuous actions $\beta^{\tau}$ and $\Ind(\sigma)$ in the following sense:
$$\forall F \in \rL^\infty(G/H) \ovt N, \forall g \in G, \quad \Ind(\sigma)_g(\pi^{\tau}(F)) = \pi^{ \tau}((\beta^{\tau})_g(F)).$$
Assume that $\psi \in N_\ast$ is a $H$-invariant normal state. Then the formula $(\nu_{G/H} \otimes \psi) \circ (\pi^{\tau})^{-1}$ defines a normal state on $\Ind_H^G(N)$ that does {\em not} depend on the choice of the measurable section $\tau$. We call it the {\em canonical} normal state on $\Ind_H^G(N)$ associated with $\nu_{G/H}$ and $\psi$ and simply denote it by $\nu_{G/H} \otimes \psi$. Observe that if $\nu_{G/H}$ is moreover $G$-invariant, then  $\nu_{G/H} \otimes \psi$ is $G$-invariant. When identifying $\Ind_H^G(N)$ with $\rL^\infty(G/H) \ovt N$, this state is given by the formula
\[ \forall f \in \Ind_H^G(N), \quad (\nu_{G/H} \otimes \psi)(f) = \int_{G/H} \psi(f(w)) \dd\nu_{G/H}(w).\]

We will use this general framework in the following two concrete situations. Assuming that $G$ is a connected semisimple Lie group with finite center, we will consider the case when 
\begin{itemize}
\item [$(\rm i)$] $H = \Gamma$ is a lattice subgroup of $G$. In that case, we denote by $m_{G/\Gamma} \in \Prob(G/\Gamma)$ the unique $G$-invariant Borel probability measure on $G/\Gamma$. This situation will appear in Sections \ref{section:induced} and \ref{section:results}.

\item [$(\rm ii)$] $H = P$ is a minimal parabolic subgroup of $G$. In that case, we denote by $\nu_P \in \Prob(G/P)$ the unique $K$-invariant Borel probability measure on $G/P$, where $K < G$ is a fixed maximal compact subgroup so that $G = KP$. Then $\nu_P$ is $G$-quasi-invariant. This situation will appear in Section \ref{section:NZ}.
\end{itemize}

\subsection{Structure theory of $G/P$}\label{subsection:G/P}

We follow the exposition given in \cite[Section 2]{NZ00} (see also \cite[I.1.1, I.1.2, II.\S 3]{Ma91} for further details). Let $G$ be any noncompact connected semisimple Lie group with trivial center and no nontrivial compact factor. Then $G \cong \Ad(G)$ may be viewed as the set of real points of a real algebraic group $\mathbf G$, see for instance \cite[Proposition 3.1.6]{Zi84}.

Let $S < G$ be a maximal $\R$-split torus and $\Phi$ the corresponding root system. We denote by $\Phi^+$ a choice of positive roots, and by $\Delta$ the corresponding set of simple positive roots. For every $\theta \subset \Delta$, denote by $P_\theta < G$ the corresponding parabolic subgroup as described in \cite[Section 2]{NZ00}. We denote by $V_\theta$ its unipotent radical, and by $R_\theta$ its reductive part, which are both connected. We then have the Levi decomposition $P_\theta = R_\theta \ltimes V_\theta$. The center $S_\theta$ of $R_\theta$ is also connected, and $R_\theta$ coincides with the centralizer of $S_\theta$ inside $G$.
We also have the opposite parabolic subgroup $\overline  P_\theta = R_\theta \ltimes \overline V_\theta$ where $\overline V_\theta$ is the unipotent radical of $\overline P_\theta$. We have $P_\theta \cap \overline V_\theta = \{e\}$ and $P_\theta \cap \overline P_\theta = R_\theta$. If $\theta = \emptyset$, then we have $S_\emptyset = S$, and we set $V = V_\emptyset$ and $P = P_\emptyset$, where $P < G$ is a minimal parabolic subgroup. If $\theta = \Delta$, then we have $S_\Delta = \{e\}$, $V_\Delta = \{e\}$ and $P_\Delta = G$.

For $\theta \subset \Delta$, we set
\[S_\theta' = \left\{s \in Z(S_\theta) \mid \Int(s) \text{ is contracting on } V_\theta \quad \text{and} \quad \Int(s)^{-1} \text{ is contracting on } \overline V_\theta \right\}.\]
Note that $S_\theta' \neq \emptyset$ if and only if $\theta \subsetneq \Delta$ if and only if $P_\theta \subsetneq G$. Define also $\overline U_\theta = R_\theta \cap \overline V = P_\theta \cap \overline V$ and observe from the Levi decomposition of $\overline{P_\theta}$ that the product map $\overline V_\theta \rtimes \overline U_\theta \to \overline V : (\overline{v}_\theta,\overline{u}_\theta)  \mapsto \overline{v}_\theta \overline{u}_\theta $ is an isomorphism. This decomposition is nontrivial if and only if $\emptyset \subsetneq \theta \subsetneq \Delta$, if and only if $P \subsetneq P_\theta \subsetneq G$, which can only happen for groups of real rank at least two. 

\subsection{Background on continuous affine actions and Poisson boundaries}\label{subsection:poisson}
We follow the exposition given in \cite[Section 2.VI]{BS04}. Let $G$ be any lcsc group and $\mu \in \Prob(G)$ any admissible Borel probability measure. Following \cite{Fu62a, BS04}, we say that a bounded function $f : G \to \C$ is (right) $\mu$-{\em harmonic} if $f(g) = \int_G f(gh) \, {\rm d}\mu(h)$ for every $g \in G$. Any bounded $\mu$-harmonic function is necessarily continuous and we denote by $\Har^\infty(G, \mu)$ the space of all bounded (right) $\mu$-harmonic functions. On the probabilistic side, the measure $\mu$ gives rise to a random walk on $G$, whose asymptotic behaviour is encoded by the so-called {\it Poisson boundary} of $(G,\mu)$: it is a measure space $(B,\nu_B)$, where $\nu_B$ is $\mu$-stationary, i.e.\ a $(G,\mu)$-space. Although the Poisson boundary can be concretely constructed in probabilistic terms, the following property characterises it among all $(G,\mu)$-spaces.

\begin{thm}[{\cite[Theorem 2.11]{BS04}}]\label{thm:harmonicity}
The linear map $\pi : \rL^\infty(B, \nu_B) \to \Har^\infty(G, \mu) : f \mapsto \pi(f)$, where $\pi(f)(g) = \int_B f(gw) \, {\rm d}\nu_B(w)$ for every $g \in G$, is isometric, $G$-equivariant and bijective.
\end{thm}

In other words, the Poisson boundary is the only $(G,\mu)$ space for which the Poisson transform is an onto-isomorphism. 
Concrete computations of the Poisson boundary can sometimes be made. For example, if $G$ is a semisimple Lie group with a minimal parabolic subgroup $P$ and a maximal compact subgroup $K$, then the Poisson boundary of $(G,\mu)$, where $\mu$ is any left $K$-invariant admissible measure is equal to $G/P$ with its unique $K$-invariant measure, \cite{Fu62a}.

Let $E$ be any separable continuous isometric Banach $G$-module. Let $\mathcal C \subset E^*$ be any non-empty $G$-invariant weak$^*$-compact convex subset, endowed with the corresponding weak$^*$-continuous affine $G$-action. We simply say that $\mathcal C$ is a compact convex affine $G$-space. We denote by $\bary: \Prob(\cC) \to \cC : \nu \mapsto \bary(\nu)$ the {\em barycenter} map defined by $f(\bary(\nu)) = \int_{\cC} f(c) \, {\rm d}\nu(c)$ for every continuous affine function $f \in \Aff(\mathcal C)$. We define the notion of stationary point in the setting of compact convex affine $G$-spaces.

\begin{defn}\label{defn:stationary}
Let $\mathcal C$ be any compact convex affine $G$-space and $c \in \cC$ any point. Denote by $ \mu_c \in \Prob(\mathcal C)$ the push-forward measure of $\mu$ under the continuous orbit map $G \to \mathcal C : g \mapsto gc$. We say that $c$ is {\em $\mu$-stationary} if $c = \bary( \mu_c)$.
\end{defn}

Concretely, an element $c \in \mathcal C$ is $\mu$-stationary if and only if
\begin{equation}\label{stationary element} \forall f \in \Aff(\mathcal C), \quad f(c) = 
\int_G f(gc) \, {\rm d}\mu(g).\end{equation}
It follows from a standard averaging argument that the subset $\mathcal C_\mu \subset \mathcal C$ of $\mu$-stationary points is not empty (see e.g.\ \cite[Lemma 1.2]{Fu62b}).

\begin{exmp}\label{exmp:C*1}
Assume that $E = A$ is a separable unital $\rC^*$-algebra and $\sigma : G \curvearrowright A$ is a C*-action. Denote by $\mathcal C = \mathcal S(A)$ the state space of $A$ and consider the corresponding weak$^*$-continuous affine action $G \curvearrowright \mathcal S(A)$. Then for every state $\phi \in \mathcal S(A)$, $\phi$ is $\mu$-stationary in the sense of Definition \ref{defn:stationary} if and only if $\phi$ is stationary in the sense of Subsection \ref{subsection:vN}. 
\end{exmp}

Let us point out that the above example has no von Neumann analogue: the convex set of normal states on a von Neumann algebra is not weak$^*$-closed in general. It is in fact weak$^*$-dense in the whole state space.

We will need the existence and essential uniqueness of the Poisson boundary map in the context of compact convex affine $G$-spaces (see \cite{Fu62b}).

\begin{thm}[{\cite[Theorem 2.16]{BS04}}]\label{thm:poisson-map}
Denote by $(B, \nu_B)$ the $(G, \mu)$-Poisson boundary. Let $\mathcal C$ be any compact convex  affine $G$-space and $c \in \cC$ any $\mu$-stationary point. Then there exists an essentially unique $G$-equivariant measurable map $\beta_c: B \to \cC$ such that $\bary((\beta_c)_\ast \nu_B) = c$. We say that $\beta_c$ is the {\em boundary} map. 
\end{thm}

In the above theorem, the $G$-equivariance of the map $\beta_c: B \to \cC$ is only meant almost everywhere: for every $g \in G$, for almost every $w \in B$, $\beta_c(g w) = g\beta_c(w)$. Since $\cC$ is a standard Borel space, it follows from \cite[Proposition B.5]{Zi84} that there exists a $G$-invariant conull measurable subset $X \subset B$ and a measurable map $\alpha : X \to \mathcal C$ such that $\alpha$ and $\beta_c$ coincide $\nu_B$-almost everywhere and $\alpha (g w) = g \alpha(w)$ for every $g \in G$ and every $w \in X$. In particular, when $B$ is a transitive $G$-space we must have $X = B$ and hence $\alpha$ is everywhere defined and $G$-equivariant. Note that there exists a unique such map which coincides almost everywhere with $\beta_c$.

From now on, assume that $G$ is a connected semisimple Lie group with finite center. Choose a maximal compact subgroup $K < G$ and a minimal parabolic subgroup so that $G = KP$. Denote by $\nu_P \in \Prob(G/P)$ the unique $K$-invariant Borel probability measure. Choose a left $K$-invariant admissible Borel probability measure $\mu \in \Prob(G)$. Since $\mu \ast \nu_P$ is $K$-invariant, it follows that $\mu \ast \nu_P = \nu_P$ and hence $\nu_P$ is $\mu$-stationary. Let $\mathcal C$ be any compact convex affine $G$-space. Denote by $\mathcal C_P \subset \mathcal C$ the compact convex subset of $P$-fixed points and by $\mathcal C_\mu \subset \mathcal C$ the compact convex subset of $\mu$-stationary points. Note that the set $\mathcal C_P$ is not empty since $P$ is amenable. For $b \in \mathcal C_P$, we define $\iota(b) = \int_{G/P} gb \, {\rm d}\nu_P(gP)$. In other words, $\iota(b)$ is the barycenter of the push-forward measure of $\nu_P$ under the continuous map $ G/P \to \mathcal C : gP \mapsto gb$.

\begin{thm}[{\cite[Lemma 2.1, Theorem 2.1]{Fu62b}}]\label{thm:representation}
The map $\iota : \mathcal C_P \to \mathcal C_\mu$ is a weak$^*$-continuous affine isomorphism.
\end{thm}

\begin{proof}
Strictly speaking, \cite[Theorem 2.1]{Fu62b} only covers the case where $\cC$ is of the form $\cC = \Prob(M)$ for some compact $G$-space $M$. One can check that the same proof applies in the general case. Alternatively, one can deduce the general case from Furstenberg's statement simply by using the maps $\delta: \rC \to \Prob(\cC) : c  \mapsto \delta_c$ and $\bary: \Prob(\cC) \to \cC : \nu \mapsto \bary(\nu)$.

The key fact in the proof is that the Poisson boundary of $(G,\mu)$ is $(G/P,\nu_P)$ (see \cite[Theorem 5.3]{Fu62a}). In this way, Theorem \ref{thm:poisson-map} and the paragraph following it imply that $\cC_\mu$ is affinely homeomorphic with the space of everywhere $G$-equivariant maps $G/P \to \cC$. Further, evaluation of such $G$-maps at the point $P \in G/P$ gives an affine homeomorphism with $\cC_P$. One checks that the formula for $\iota$ gives exactly the inverse of this affine homeomorphism $\cC_\mu \to \cC_P$.
\end{proof}

We infer the following useful result.

\begin{cor}\label{exmp:C*2}
Assume that $E = A$ is a separable unital $\rC^*$-algebra and $\sigma : G \curvearrowright A$ is a continuous action. Denote by $\mathcal C = \mathcal S(A)$ the state space of $A$ and consider the corresponding weak$^*$-continuous affine action $G \curvearrowright \mathcal S(A)$. Using Theorem \ref{thm:representation}, for every $\mu$-stationary state $\phi \in \mathcal S(A)$, there exists a unique $P$-invariant state $\psi \in \mathcal S(A)$ such that $$\phi = \int_{G/P} g\psi \, {\rm d}\nu_P(gP) = \int_{G/P} \psi \circ \sigma_g^{-1} \, {\rm d}\nu_P(gP).$$
\end{cor}

\section{Absolute continuity for states on C*-algebras}
\label{Section AC}

 In this section, we investigate the noncommutative analogue of the notion of absolute continuity of Borel probability measures on a compact metrizable space.

 We fix a separable unital $\rC^*$-algebra $A$. For every $\phi \in \cS(A)$, we denote by $(\pi_\phi,H_\phi,\xi_\phi)$ the corresponding GNS triple. In other words, $H_\phi$ is a Hilbert space, $\pi_\phi: A \to \mathbf B(H_\phi)$ is a C*-morphism, and $\xi_\phi$ is a unit vector in $H_\phi$ such that $\pi_\phi(A)\xi_\phi$ is dense in $H_\phi$ and $\langle \pi_\phi(a)\xi_\phi,\xi_\phi\rangle = \phi(a)$, for every $a \in A$.
We also denote by $\phi$ the normal state $\langle \, \cdot\, \xi_\phi, \xi_\phi\rangle$ on $\pi_\phi(A)\dpr$. By construction, we have $\phi(\pi_\phi(a)) = \phi(a)$ for every $a \in A$.

\begin{defn}
Let $\phi,\psi \in \cS(A)$ be any states. We say that $\psi$ is {\em absolutely continuous} with respect to $\phi$ and write $\psi \prec \phi$, if for every $a \in A$, we have $\|\pi_\psi(a)\| \leq \|\pi_\phi(a)\|$ and the well-defined unital $\ast$-homomorphism $\pi_\phi(A) \to \pi_\psi(A)\dpr : \pi_\phi(a) \mapsto \pi_\psi(a)$ has a normal extension to $\pi_\phi(A)\dpr$, denoted by $\pi_{\phi,\psi}$.
\end{defn}

We say that $\phi$ and $\psi$ are equivalent and write $\phi \sim \psi$ if $\phi \prec \psi$ and $\phi \prec \psi$. In this case, $\pi_{\phi, \psi} : \pi_\phi(A)\dpr \to \pi_\psi(A)\dpr$ is an onto von Neumann algebra isomorphism.

Following \cite[Section 1.4]{BO08}, every C*-representation $\pi : A \to \mathbf B(H_\pi)$ admits a normal extension $\tpi: A^{\ast \ast} \to \pi(A)\dpr$. Here the bidual space $A^{**}$ is viewed as the universal enveloping von Neumann algebra of $A$. Recall that the image of a von Neumann algebra under a normal $\ast$-homomorphism is again a von Neumann algebra. So $\tpi: A^{\ast \ast} \to \pi(A)\dpr$ is necessarily onto.

We denote by $z_\pi$ the support projection of $\tpi$, i.e.\ the smallest projection in $A^{**}$ such that $\tpi(z_\pi) = 1$. In particular $z_\pi$ belongs to $\cZ(A^{**})$ and $\tpi$ restricts to an onto von Neumann isomorphism $A^{\ast \ast}z_\pi \to \pi(A)\dpr : a z_\pi \mapsto \tpi(a)$. For every $\phi \in \mathcal S(A)$, we simply write $z_\phi = z_{\pi_\phi}$. We note the following characterization.

\begin{lem}\label{lem:abs-cont}
Let $\phi,\psi \in \cS(A)$ be any states. The following conditions are equivalent:
\begin{itemize}
\item [$(\rm i)$] $\psi \prec \phi$.
\item [$(\rm ii)$] For every $a \in A$, we have $\|\pi_\psi(a)\| \leq \|\pi_\phi(a)\|$ and the well-defined state $\pi_\phi(A) \to \C : \pi_\phi(a) \mapsto \psi(a)$ has a normal extension to $\pi_\phi(A)\dpr$.
\item [$(\rm iii)$] $z_\psi \leq z_\phi$.
\end{itemize}
\end{lem}

\begin{proof}
$(\rm i) \Rightarrow (\rm ii)$ is obvious. 

$(\rm iii) \Rightarrow (\rm i)$. We obtain the desired normal unital $\ast$-homomorphism as the following composition:
\[\begin{matrix}
\pi_\phi(A)'' & \to  &  A^{**}z_\phi & \to & A^{**}z_\psi & \to & \pi_\psi(A)''\\
\tpi_\phi(a) & \mapsto & az_\phi & \mapsto & az_\psi & \mapsto & \tpi_\psi(a)
\end{matrix} \qquad (a \in A^{**}).\]

$(\rm ii) \Rightarrow (\rm iii)$. Denote by $\rho_{\phi, \psi} : \pi_\phi(A)\dpr \to \C : \pi_\phi(a) \mapsto \psi(a)$ the corresponding normal state. By definition, we have $\rho_{\phi, \psi} \circ \widetilde \pi_\phi = \psi \circ \widetilde \pi_\psi$ on $A^{\ast \ast}$ This implies that $\psi(\widetilde \pi_\psi(z_\phi)) = 1$ and so $\widetilde \pi_\psi(z_\phi) \xi_\psi = \xi_\psi$. For every $a, b \in A$, we have
\begin{align*}
\langle \widetilde \pi_\psi(z_\phi) \, \pi_\psi(a) \xi_\psi, \pi_\psi(b) \xi_\psi\rangle &=  \langle \pi_\psi(a) \, \widetilde \pi_\psi(z_\phi)  \xi_\psi, \pi_\psi(b) \xi_\psi\rangle \\
&=  \langle \pi_\psi(a)  \xi_\psi, \pi_\psi(b) \xi_\psi\rangle
\end{align*}
and so $\widetilde \pi_\psi(z_\phi) = 1$. By construction, this further implies that $z_\psi \leq z_\phi$.
\end{proof}

The following proposition will play a central role in our analysis.

\begin{prop}\label{ac crit}
Let $\nu_1,\nu_2 \in \Prob(\cS(A))$ be any Borel probability measures. Assume that $\nu_2$ is absolutely continuous with respect to $\nu_1$. Then $\phi_2 = \bary(\nu_2)$ is absolutely continuous with respect to $\phi_1 = \bary(\nu_1)$.

Assume moreover that $\nu_2$ and $\nu_1$ are equivalent, so that $\phi_2$ and $\phi_1$ are equivalent. Then $\phi_1$ is faithful on $\pi_{\phi_1}(A)\dpr$ if and only if $\phi_2$ is faithful $\pi_{\phi_2}(A)\dpr$.
\end{prop}

Before proving the proposition, let us discuss GNS representations associated with states expressed as barycenters of measures. We refer to \cite[Sections IV.6 and IV.8]{Ta02} for more on this. Let $\nu \in \Prob(\cS(A))$ be any Borel probability measure and denote by $\phi = \bary(\nu) \in \mathcal S(A)$.  We introduce the direct integral unitary representation $(\pi_\nu,H_\nu, \xi_\nu)$ of $A$ associated with the measurable field $\{\pi_\psi \mid \psi \in \mathcal S(A)\}$:
\[H_\nu = \int_{\mathcal S(A)}^\oplus H_\psi \, {\rm d}\nu(\psi), \quad \pi_\nu = \int_{\mathcal S(A)}^\oplus \pi_\psi \, {\rm d}\nu(\psi), \quad \xi_\nu = \int_{\mathcal S(A)}^\oplus \xi_\psi \, {\rm d}\nu(\psi).\]
Observe that for all $a,b \in A$, the map $\psi \mapsto \langle a\xi_\psi,b\xi_\psi\rangle = \psi(b^*a)$ is continuous hence measurable on $\mathcal S(A)$, and the above direct integrals all make sense.
Observe that $\rL^\infty(\mathcal S(A),\nu)$ acts diagonally on $H_\nu$. We need the following useful result.

\begin{lem}[{\cite[Lemma 4.1]{AB18}}]\label{GNS identification}
The map $\pi_\nu(A)\dpr \to \pi_\phi(A)\dpr : \pi_\nu(a) \mapsto \pi_\phi(a)$ extends to a well-defined surjective normal unital $\ast$-isomorphism. In other words, we have $z_{\pi_\nu} = z_\phi$.
\end{lem}
\begin{proof}
Set $M = \pi_\nu(A)\dpr$ and observe that $M \subset \int_{\mathcal S(A)}^\oplus \pi_\psi (A)\dpr\, {\rm d}\nu(\psi)$. Denote by $p \in M' \cap \mathbf B(H_\nu)$ the orthogonal projection onto the closed subspace $K_\nu = \overline{\pi_\nu(A)\xi_\nu}$. Observe that $\xi_\nu$ is a $\pi_\nu(A)$-cyclic vector in $K_\nu$ that implements the state $\phi$ on $A$. So by uniqueness of the GNS representation, the representation $A  \to \mathbf B(\cK_\nu) : a \mapsto \pi_\nu(a) p$ is unitarily conjugate to $\pi_\phi$. In particular, it indeed extends to a surjective normal $\ast$-isomorphism $\pi_\phi(A)\dpr \to M p : \pi_\phi(a) \mapsto \pi_\nu(a) p$.

We are left to check that the normal unital $\ast$-homomorphism $ M \to Mp :x \mapsto xp$ is injective. Let $x \in M$ such that $x p = 0$. Since $x \in M$, $x$ commutes with $\rL^\infty(\mathcal S(A),\nu)$, and so $x \in \int_{\mathcal S(A)}^\oplus \pi_\psi(A)\dpr \, {\rm d}\nu(\psi)$ may be disintegrated $x = \int_{\mathcal S(A)}^\oplus x_\psi \, {\rm d}\nu(\psi)$ where $x_\psi \in \pi_\psi(A)\dpr$ for $\nu$-almost every $\psi \in \mathcal S(A)$. If $xp = 0$, we derive that $\int_{\mathcal S(A)}^\oplus x_\psi \pi_\psi(a) \xi_\psi \, {\rm d}\nu(\psi) = 0$, for all $a \in A$. Hence for every $a \in A$ and $\nu$-almost every $\psi \in \mathcal S(A)$, we have $x_\psi \pi_\psi(a)\xi_\psi = 0$. Since $A$ is separable, for $\nu$-almost every $\psi \in \mathcal S(A)$ and every $a \in A$, we have $x_\psi \pi_\psi(a)\xi_\psi = 0$. Since each $\xi_\psi$ is $\pi_\psi(A)$-cyclic, we conclude that $x_\psi = 0$ for $\nu$-almost every $\psi \in \mathcal S(A)$, i.e.\ $x = 0$.
\end{proof}

\begin{proof}[Proof of Proposition \ref{ac crit}]
Set $X = \mathcal S(A)$. Denote by $f = \frac{{\rm d}\nu_2}{{\rm d}\nu_1} \in \rL^1(X, \nu_1)$ the Radon--Nikodym derivative of $\nu_2$ with respect to $\nu_1$. Observe that the vector $\xi_2 = \int_X^\oplus f(\psi)^{1/2}\xi_\psi \, {\rm d}\nu_1(\psi) \in H_{\nu_1}$ is a unit vector such that
\[\forall a \in A, \quad \langle \pi_{\nu_1}(a)\xi_2,\xi_2\rangle = \phi_2(a).\]
Thus, the state $\pi_{\nu_1}(A) \to \C : \pi_{\nu_1}(a) \mapsto \phi_2(a)$ has a normal extension to $\pi_{\nu_1}(A)\dpr$. The first part of the proposition now follows from Lemma \ref{GNS identification}.

For the second part of the proposition, assume that $\nu_2$ and $\nu_1$ are equivalent and that $\phi_1$ is faithful on $\pi_{\phi_1}(A)\dpr$. Then the map $\pi_{\phi_1, \phi_2} : \pi_{\phi_1}(A)\dpr \to \pi_{\phi_2}(A)\dpr : \pi_{\phi_1}(a) \mapsto \pi_{\phi_2}(a)$ is a surjective normal unital $\ast$-isomorphism. We show that $\phi_2 \circ \pi_{\phi_1, \phi_2}$ is faithful on $\pi_{\phi_1}(A)\dpr$. This will imply that $\phi_2$ is faithful on $\pi_{\phi_2}(A)\dpr$. Observe that 
\begin{equation}\label{eq:state}
\forall a \in A, \quad (\phi_2 \circ \pi_{\phi_1, \phi_2})(\pi_{\phi_1}(a)) = \phi_2(a) = \langle \pi_{\nu_1}(a)\xi_2,\xi_2\rangle.
\end{equation}
Set $M = \pi_{\nu_1}(A)\dpr$. Since $\phi_1$ is faithful on $\pi_{\phi_1}(A)\dpr$, Lemma \ref{GNS identification} and its proof imply that the unit vector $\xi_{1} = \int_X^\oplus \xi_\psi \, {\rm d}\nu_1(\psi)  \in H_{\nu_1}$ is $M$-separating, meaning that for every $x \in M$, if $x \xi_{1} = 0$ then $x = 0$. In order to show that $\phi_2 \circ \pi_{\phi_1, \phi_2}$ is still faithful on $\pi_{\phi_1}(A)\dpr$, using Lemma \ref{GNS identification} and \eqref{eq:state}, it suffices to show that the unit vector $\xi_2 \in H_{\nu_1}$ is $M$-separating. Let $x \in M$ be such that $x\xi_2 = 0$. Since $x \in M$, $x$ commutes with $\rL^\infty(X,\nu_1)$, and so $x$ may be written $x = \int_{X}^\oplus x_\psi \, {\rm d}\nu_1(\psi)$ where $x_\psi \in \pi_\psi(A)\dpr$ for $\nu_1$-almost every $\psi \in X$. Since $x\xi_2 = 0$, we have $f(\psi)^{1/2} x_\psi \xi_\psi = 0$ for $\nu_1$-almost every $\psi \in X$. Since $\nu_2$ and $\nu_1$ are equivalent, we have $f(\psi) > 0$ for $\nu_1$-almost every $\psi \in X$ and so $ x_\psi \xi_\psi = 0$ for $\nu_1$-almost every $\psi \in X$. This implies that $x\xi_1 = 0$ and so $x = 0$.
\end{proof}

\section{Induced actions and stationary states}\label{section:induced}

In this section, we explain how to induce stationary states in the context of Furstenberg measures. The approach presented here is rather pedestrian, but there are more synthetic ways of presenting the construction by using the Poisson transform and harmonic functions (via Theorem \ref{thm:harmonicity}), see \cite[Section 4.1]{BBHP20}.

Let us recall the notation we will use. Here, $G$ will be any connected semisimple Lie group with finite center, with a maximal compact subgroup $K$ and a minimal parabolic subgroup $P$, so that $G = KP$. We denote by $\nu_P \in \Prob(G/P)$ the unique $K$-invariant Borel probability measure on $G/P$, and by $\mu \in \Prob(G)$ a $K$-invariant admissible measure. Since $\mu \ast \nu_P$ is $K$-invariant, we have $\mu \ast \nu_P = \nu_P$, i.e.\ $\nu_P$ is $\mu$-stationary. This implies that $\nu_P$ is $G$-quasi-invariant (see \cite[Lemma 1.1]{NZ97}). Let $\Gamma< G$ be any lattice. By \cite[Theorem 3]{Fu67}, there exists a {\em Furstenberg} probability measure $\mu_0 \in \Prob(\Gamma)$, whose support coincides with $\Gamma$ and for which $(G/P, \nu_P)$ is the $(\Gamma, \mu_0)$-Poisson boundary \cite{Fu62a}.
 
Let $(M, \phi)$ be any $(\Gamma, \mu_0)$-von Neumann algebra with separable predual. We assume that $\phi \in M_\ast$ is a faithful normal state. We will construct a $\mu$-stationary faithful normal state $\varphi$ on the induced $G$-von Neumann algebra $\Ind_\Gamma^G(M)$. 

We denote by $(\pi_\phi, H_\phi, \xi_\phi)$ the corresponding GNS triple. Since $\phi \in M_\ast$ is faithful, $\pi_\phi$ is faithful. 
Therefore, we may identify $M$ with $\pi_\phi(M)$ and assume that $\pi_\phi = \id$. Choose a globally $\Gamma$-invariant ultraweakly dense separable unital $\rC^*$-subalgebra $A \subset M$. For example, start with any separable dense $\rC^*$-subalgebra $A_0$ of $M$ and choose $A = \rC^*(\Gamma \cdot A_0)$. 
Regarding $\phi|_A \in \mathcal S(A)$ as a $\mu_0$-stationary state on $A$, Theorem \ref{thm:poisson-map} yields a $\Gamma$-equivariant $\nu_P$-measurable boundary map $\beta_\phi : G/P \to \mathcal S(A) : w \mapsto \phi_{w}$ such that $\phi = \bary((\beta_\phi)_\ast \nu_P)$. Set $\nu = (\beta_\phi)_\ast \nu_P \in \Prob(\mathcal S(A))$. From our observations on absolutely continuous states, we deduce the following lemma. 

\begin{lem}\label{lem:key}
For every $g \in G$, set $\nu_g = (\beta_\phi)_\ast(g_\ast \nu_P) \in \Prob(\mathcal S(A))$ and $\phi_g = \bary(\nu_g) \in \mathcal S(A)$. Then $\phi_g$ extends to faithful normal state on $M$ that we still denote by $\phi_g \in M_\ast$. 

Moreover, for every $y \in M$, the bounded map $G \to \C : g \mapsto \phi_g(y)$ is $\mu$-harmonic hence continuous. In particular, the map $G \to \mathcal S(A) : g \mapsto \phi_g$ is weak$^*$-continuous.
\end{lem}
\begin{proof}
Since $\nu_P \in \Prob(G/P)$ is $G$-quasi-invariant, for every $g \in G$, $g_\ast \nu_P$ and $\nu_P$ are equivalent probability measures on $G/P$. This implies that for every $g \in G$, $\nu_g = (\beta_\phi)_\ast(g_\ast \nu_P)$ and $\nu = (\beta_\phi)_\ast \nu_P$ are equivalent probability measures on $\mathcal S(A)$. Since $\phi \in M_\ast$ is faithful, Proposition \ref{ac crit} implies that for every $g \in G$, $\phi_g$ extends to faithful normal state on $M = \pi_\phi(M)$. Moreover, by definition of the state $\phi_g = \bary((\beta_\phi)_\ast (g_\ast \nu_P)) \in \mathcal S(A)$, we have
\[ \forall g \in G, \forall a \in A, \quad \phi_g(a) = \int_{G/P} \phi_{gw}(a) \, {\rm d}\nu_P(w).\]
Since $\nu_P$ is $\mu$-stationary, it follows that the bounded function $G \to \C : g \mapsto \phi_g(a)$ is $\mu$-harmonic for all $a \in A$, that is, $\phi_{g}(a) = \int_G \phi_{gh}(a) \, {\rm d}\mu(h)$. Let $y \in M$ be any element. By Kaplansky's density theorem and since $M$ has separable predual, we may choose a sequence $(a_n)_{n \in \N}$ in $A$ such that $\sup_{n \in \N} \|a_n\| \leq \|y\|$ and $a_n \to y$ strongly in $M$ as $n \to \infty$. For every $g \in G$, since $\phi_g \in M_\ast$, we have $\lim_n \phi_g(a_n) = \phi_g(y)$. Lebesgue's dominated convergence theorem implies that 
\begin{equation}\label{eq:harmonic}
\forall g \in G, \forall y \in M, \quad \phi_g (y)= \int_G \phi_{gh}(y) \, {\rm d}\mu(h).
\end{equation}
 Therefore the bounded map $G \to \C : g \mapsto \phi_g(y)$ is $\mu$-harmonic hence continuous.  In particular, the map $G \to \mathcal S(A) : g \mapsto \phi_g$ is weak$^*$-continuous.
\end{proof}

In the above lemma, $\phi_g$ should be thought as a translate $g\phi$ of $\phi$ although there is no $G$-action on $M$. This simulation of a $G$-action will be used to untwist the cocycle that appears in the induction formula, and thus get a well-behaved state upstairs.

Choose a measurable section $\tau : G/\Gamma \to G$ and define the measurable $1$-cocycle $c: G \times G/\Gamma \to \Gamma$ by the formula $c(g,x) = \tau (gx)^{-1}g \tau(x) $ for every $g \in G$, $x \in G/\Gamma$. Denote by $m_{G/\Gamma} \in \Prob(G/\Gamma)$ the unique $G$-invariant Borel probability measure. Regard $\rL^\infty(G/\Gamma) \ovt M$ as the von Neumann algebra $\rL^\infty(G/\Gamma, M)$ of all essentially bounded measurable functions $F : G/\Gamma \to M$, modulo equality $m_{G/\Gamma}$-almost everywhere. Set $\mathcal M = \Ind_\Gamma^G(M)$. Recall from Subsection \ref{subsection:vN} that we may identify $\mathcal M$ with $\rL^\infty(G/\Gamma) \ovt M$. Under this identification, the continuous induced action $\Ind(\sigma) : G \curvearrowright \mathcal M$ is given by
\[\forall f \in \mathcal M, \forall g \in G, \quad \Ind(\sigma)_g (f)(x) = \sigma_{c(g,g^{-1}x)}(f(g^{-1}x)).\]
For simplicity, we will denote the induced action $\Ind(\sigma)$ by $\tsigma$.

\begin{lem}\label{lem:measure}
Let $f \in \mathcal M$ be any element. Then the function $G \times G/\Gamma \to \C : (g, x) \mapsto \phi_{g}(f(x))$ is measurable. In particular, the functions $(g, x) \mapsto \phi_{\tau(x)^{-1}g}(f(x))$ and $x \mapsto \phi_{\tau(x)^{-1}}(f(x))$ are measurable on $G \times G/\Gamma$ and $G/\Gamma$, respectively.
\end{lem}

\begin{proof}
Denote by $\mathcal A \subset \mathcal M$ the ultraweakly dense unital $\ast$-subalgebra consisting of all finite sums of elements in $\mathcal M$ of the form $1_E \otimes y$ where $E \subset G/\Gamma$ is a measurable subset and $y \in M$. Lemma \ref{lem:key} implies that for every $a \in \mathcal A \subset \mathcal M$, the map $G \times G/\Gamma \to \C : (g, x) \mapsto \phi_{g}(a(x))$ is measurable. 

Let now $f \in \mathcal M$ be any element. By Kaplansky's density theorem and since $\mathcal M$ has separable predual, we may choose a sequence $(a_n)_{n \in \N}$ in $\mathcal A$ such that $\sup_{n \in \N} \|a_n\|_\infty \leq \|f\|_\infty$ and $a_n \to f$ strongly in $\mathcal M$ as $n \to \infty$. Up to choosing a subsequence, there exists a $m_{G/\Gamma}$-conull measurable subset $X_f \subset G/\Gamma$ such that for every $x \in X_f$, we have $\sup_{n \in \N} \|a_n(x)\| \leq \|f\|_\infty$ and $a_n(x) \to f(x)$ strongly in $ M$ as $n \to \infty$ (see \cite[Exercise IV.8.2]{Ta02}). For every $(g, x) \in G \times X_f$, since $\phi_{g} \in M_\ast$, we have $\lim_n \phi_{g}(a_n(x)) = \phi_{g}(f(x))$. This implies that the map $G \times X_f \to \C : g \mapsto \phi_{g}(f(x))$ is measurable since it is the pointwise limit of the sequence of measurable maps $G \times X_f \to \C : (g, x) \mapsto \phi_{g}(a_n(x))$. This gives the first statement. The rest follows from the measurability of $\tau$.
\end{proof}

We are now ready to state and prove the main result of this section.

\begin{thm}\label{induced stationary state}
Keep the same notation as above. Then the map $\varphi: \cM \to \C$ defined by the following formula is a $\mu$-stationary faithful normal state on $\mathcal M$:
\[\forall f \in \mathcal M, \quad \varphi(f) = \int_{G/\Gamma} \phi_{\tau(x)^{-1}}(f(x)) \, {\rm d}m_{G/\Gamma}(x).\] Moreover, $\varphi$ is $G$-invariant if and only if $\phi$ is $\Gamma$-invariant.
\end{thm}
\begin{proof}
Lemma \ref{lem:measure} justifies that the formula defining $\varphi$ makes sense. For this, denote by $\Ball_{\mathcal M}(0, 1)$ the closed ball in $\mathcal M$ of center $0$ and radius $1$ with respect to the uniform norm. Let us first check that the state $\varphi$ is {\em normal}. Let $f \in \Ball_{\mathcal M}(0, 1)$ be any element and $(f_n)_n$ any sequence in $\Ball_{\mathcal M}(0, 1)$ such that $f_n \to f$ strongly in $\mathcal M$. By contradiction, assume that the sequence $(\varphi(f_n))_{n \in \N}$ does not converge to $\varphi(f)$. Up to choosing a subsequence, we may assume that there exists $\varepsilon > 0$ such that $\inf_{n \in \N} |\varphi(f) - \varphi(f_n)| \geq \varepsilon$. Up to choosing a further subsequence, we may assume that for $m_{G/\Gamma}$-almost every $x \in G/\Gamma$, we have $f_n(x) \to f(x)$ strongly in $M$ as $n \to \infty$ (see \cite[Exercise IV.8.2]{Ta02}). Since $\phi_{\tau(x)^{-1}} \in M_\ast$ for every $x \in G/\Gamma$, we have   $\lim_n \phi_{\tau(x)^{-1}}(f_n(x)) = \phi_{\tau(x)^{-1}}(f(x))$ for $m_{G/\Gamma}$-almost every $x \in G/\Gamma$. Lebesgue's dominated convergence theorem implies that $\lim_n \varphi(f_n) = \varphi(f)$. This is a contradiction. Therefore, $\varphi$ is a normal state.

Observe that if $\varphi(f^*f) = 0$, then for $m_{G/\Gamma}$-almost every $x \in G/\Gamma$, $\phi_{\tau(x)^{-1}}(f(x)^*f(x)) = 0$. Since $\phi_{\tau(x)^{-1}}$ is faithful for every $x \in G/\Gamma$, this implies that $f = 0$, proving that $\varphi$ is faithful.

We now prove the stationarity condition. Since the boundary map $\beta_\phi : G/P \to \mathcal S(A)$ is $\nu_P$-measurable and $\Gamma$-equivariant and since $\phi_g \in M_\ast$ for every $g \in G$, it follows that 
\begin{equation}\label{Gamma equiv}
\forall  \gamma \in \Gamma, \forall g \in G, \forall y \in M, \quad \gamma \phi_g (y)= \phi_{\gamma g}(y).
\end{equation}
Then for every $f \in \cM$ and every $g \in G$, we have
\begin{align*}
g\varphi(f) = \varphi(\tsigma_g^{-1}( f)) & =  \int_{G/\Gamma} \phi_{\tau(x)^{-1}}(\sigma_{c(g^{-1},gx)}(f(gx))) \, {\rm d}m_{G/\Gamma}(x)\\
& = \int_{G/\Gamma}  \phi_{c(g^{-1},gx)^{-1}\tau(x)^{-1}}(f(gx)) \, {\rm d}m_{G/\Gamma}(x) \quad (\text{using } \eqref{Gamma equiv})\\
& = \int_{G/\Gamma} \phi_{\tau(gx)^{-1}g}(f(gx)) \, {\rm d}m_{G/\Gamma}(x).
\end{align*}
Since $m_{G/\Gamma} \in \Prob(G/\Gamma)$ is $G$-invariant, we conclude that
\begin{equation}\label{induced shift}
g\varphi( f) = \int_{G/\Gamma} \phi_{\tau(x)^{-1}g}(f(x)) \, {\rm d}m_{G/\Gamma}(x).
\end{equation}
Combining \eqref{induced shift} with \eqref{eq:harmonic} and using Lemma \ref{lem:measure} with Fubini's theorem, for every $f \in \mathcal M$, we obtain
\begin{align*}
 \int_G g\varphi( f) \, {\rm d}\mu(g)  &=  \int_G \left(\int_{G/\Gamma} \phi_{\tau(x)^{-1}g}(f(x))  \,  {\rm d}m_{G/\Gamma}(x) \right) {\rm d}\mu(g)\\
&=  \int_{G/\Gamma} \left( \int_G \phi_{\tau(x)^{-1}g}(f(x)) \, {\rm d}\mu(g) \right) {\rm d}m_{G/\Gamma}(x) \\
&= \int_{G/\Gamma}  \phi_{\tau(x)^{-1}}(f(x)) \, {\rm d}m_{G/\Gamma}(x)\\
& = \varphi(f).  
\end{align*}
Therefore, $\varphi$ is $\mu$-stationary. 

Finally, we prove that $\varphi$ is $G$-invariant if and only if $\phi$ is $\Gamma$-invariant. First, assume that $\phi$ is $\Gamma$-invariant. The corresponding boundary map $\beta_\phi : G/P \to \mathcal S(A)$ is essentially constant. This implies that $\phi_g = \phi$ for every $g \in G$. By construction, this implies that $\varphi = m_{G/\Gamma} \otimes \phi$ and so $\varphi$ is $G$-invariant. Conversely, assume that $\varphi$ is $G$-invariant. By \eqref{induced shift}, for every $g \in G$ and every $f \in \cM$,  we have  
\[ \int_{G/\Gamma} \phi_{\tau(x)^{-1}}(f(x)) \, {\rm d}m_{G/\Gamma}(x)  = \varphi(f) = g\varphi( f) = \int_{G/\Gamma} \phi_{\tau(x)^{-1}g}(f(x)) \, {\rm d}m_{G/\Gamma}(x) .\]
Since this holds true for every element of the form $f = h \otimes a \in \mathcal M$ with $h \in \rL^\infty(G/\Gamma)$ and $a \in A$, we conclude that for every $g \in G$, every $a \in A$ and $m_{G/\Gamma}$-almost every $x \in G/\Gamma$, we have $\phi_{\tau(x)^{-1}}(a) =  \phi_{\tau(x)^{-1}g}(a)$. Since $A$ is separable, this implies that for every $g \in G$ and $m_{G/\Gamma}$-almost every $x \in G/\Gamma$, we have $\phi_{\tau(x)^{-1}} =  \phi_{\tau(x)^{-1}g}$. 
By Fubini's theorem there exists $x \in G/\Gamma$ such that for Haar-almost every $g \in G$, we have $\phi_{\tau(x)^{-1}} =  \phi_{\tau(x)^{-1}g}$. Thus, the map $g \in G \mapsto \phi_g \in \cS(A)$ is Haar-almost everywhere constant. By Lemma \ref{lem:key}, this map is continuous, hence everywhere constant. In view of the $\Gamma$-equivariance property \eqref{Gamma equiv}, this implies that the state $\phi = \phi_e \in \mathcal S(A)$ is $\Gamma$-invariant. Thus, $\phi \in M_\ast$ is $\Gamma$-invariant.
\end{proof}

\section{A noncommutative Nevo--Zimmer theorem}\label{section:NZ}

In this section, we prove a noncommutative analogue of Nevo--Zimmer's structure theorem  for stationary actions of connected semisimple Lie groups on arbitrary von Neumann algebras (see \cite[Theorem 1]{NZ00} for stationary actions on measure spaces). 

\subsection{Background on tensor-slice maps on von Neumann algebras}

In what follows, we collect a few facts about von Neumann algebras that we will use in the proof of Theorem \ref{thm:NZ}. Recall that for any von Neumann algebra $\mathcal P$, the predual $\mathcal P_\ast$ has a canonical $\mathcal P$-$\mathcal P$-bimodule structure given by 
\[\forall b, c, T \in \mathcal P, \forall \rho \in \mathcal P_\ast, \quad (b\rho c)(T) = \rho(c T b).\]
Moreover, the isometric linear map $\mathcal P \to ((\mathcal P)_\ast)^* : T \mapsto (\rho \mapsto \rho(T))$ is surjective and ultraweakly-weak$^*$-continuous. We then identify $\mathcal P = ((\mathcal P)_\ast)^*$. We denote by $\id_{\mathcal P} \in \Aut(\mathcal P)$ the identity $\ast$-automorphism of $\mathcal P$. If $\mathcal P = \rL^\infty(X)$, then we write $\id_{\mathcal P} = \id_X$.

Let $\mathcal P = \mathcal P_1 \ovt \mathcal P_2$ be any tensor product von Neumann algebra. For every $\rho_1 \in (\mathcal P_1)_\ast$ and every $\rho_2 \in (\mathcal P_2)_\ast$, we consider the maps $\id_{\mathcal P_1} \otimes \rho_2 : x_1 \ot x_2 \in \mathcal P \mapsto \rho_2(x_2)x_1 \in \mathcal P_1$ and $\rho_1 \otimes \id_{\mathcal P_2} : x_1 \ot x_2 \in \mathcal P \mapsto \rho_1(x_1)x_2 \in \mathcal P_2$. Whenever $\rho_1$ and $\rho_2$ are unital, these maps can be seen as projections, after identifying $\cP_1$ with $\cP_1 \ot \C1$ and $\cP_2$ with $\C1 \ot \cP_2$. They can be formally defined by using GNS representations of $|\rho_1|$ and $|\rho_2|$, and are then seen to be completely bounded normal maps, called the tensor-slice maps in the terminology of \cite[Section 2]{GK95}. They satisfy $\| \id_{\mathcal P_1} \otimes \rho_2\| \leq \|\rho_2\|$, $\Vert \rho_1 \ot \id_{\cP_2} \Vert \leq \Vert \rho_1 \Vert$ and 
\[\rho_1 \circ (\id_{\mathcal P_1} \otimes \rho_2) = \rho_1 \otimes \rho_2 = \rho_2 \circ (\rho_1 \otimes \id_{\mathcal P_2}).\]
Let $\rho_2 \in (\mathcal P_2)_\ast$ and assume that $(\zeta_{n})_{n \in \N}$ is a sequence in $(\mathcal P_2)_\ast$ such that $\lim_n \|\zeta_n - \rho_2\| = 0$. Then for every $T \in \mathcal P$, we have that $(\id_{\mathcal P_1} \otimes \zeta_n)(T) \to (\id_{\mathcal P_1} \otimes \rho_2)(T)$ ultraweakly in $\mathcal P_1$ as $n \to \infty$.

Assume now that $\mathcal P = \mathcal P_1 \ovt \mathcal P_2 \ovt \mathcal P_3$. By construction of the tensor-slice maps, for every $\rho_2 \in (\mathcal P_2)_\ast$ and every $\rho_3 \in (\mathcal P_3)_\ast$, we have
\begin{equation}\label{eq:tensor-slice} \forall T \in \mathcal P, \quad (\id_{\mathcal P_1} \otimes \rho_2 \otimes \rho_3)(T)  = (\id_{\mathcal P_1} \otimes \rho_2) \left((\id_{\mathcal P_1} \otimes \id_{\mathcal P_2} \otimes \rho_3)(T) \right).
\end{equation}

\subsection{A noncommutative Nevo--Zimmer theorem}

The main result of this section is the following noncommutative Nevo--Zimmer theorem. 

\begin{thm}\label{thm:NZ}
Let $G$ be any connected semisimple Lie group with finite center and no nontrivial compact factors, all of whose simple factors are of real rank at least two. Let  $\mu \in \Prob(G)$ be any $K$-invariant admissible Borel probability measure. Let $(\mathcal M, \varphi)$ by any ergodic $(G, \mu)$-von Neumann algebra. Then the following dichotomy holds.
\begin{itemize}
\item Either $\varphi$ is $G$-invariant.
\item Or there exist a proper parabolic subgroup $P \subset Q \subsetneq G$ and a $G$-equivariant normal unital $\ast$-embedding $\Theta : \rL^\infty(G/Q, \nu_Q) \to \mathcal M$ such that $\varphi \circ \Theta = \nu_Q$.
\end{itemize}
\end{thm}

Before proving Theorem \ref{thm:NZ}, let us explain the main conceptual difficulty that appears in this noncommutative setting compared to the original commutative result \cite[Theorem 1]{NZ00}. First of all, let us mention that Nevo--Zimmer had proved earlier a weaker version of their result, assuming some mixing condition on the $P$-action on the space (see \cite{NZ97}). To get rid of this mixing assumption, they had the idea of using the so-called {\em Gauss map}, which was the key novelty in \cite{NZ00}. This new argument relied on the consideration of the {\em stabilizer map} associated with an action  $G \actson (X,\mu)$, defined as the $G$-equivariant measurable map $X \to  \Sub(G) : x \mapsto \Stab(x)$. Here, $\Sub(G)$ denotes the space of closed subgroups of $G$, endowed with the Chabauty topology. Unfortunately, there is no analogue of such a stabilizer map for actions of $G$ on arbitrary von Neumann algebras, so it is hopeless to prove Theorem \ref{thm:NZ} by simply translating Nevo--Zimmer's proof in noncommutative terms. Moreover for our purposes, the analogue of the mixing condition on the $P$-action used in \cite{NZ97} is not guaranteed. We indeed need the full strength of Theorem~\ref{thm:NZ}.

To get around this issue, we will start the proof in a similar fashion to that of \cite[Theorem 1]{NZ97} until the critical point where we would need to use the Gauss map is reached. From that point on, we will use the strong dynamics of $G$ on a well chosen homogeneous space $G/P_\theta$ to be able to reduce to the commutative case by using a result by Ge--Kadison \cite{GK95} and its generalization by Str\u{a}til\u{a}--Zsid\'o \cite{SZ98}. Before starting the proof of Theorem \ref{thm:NZ}, let us give some notation and a few preliminary lemmas.

\begin{lem} 
We may assume that $\mathcal M$ has separable predual.
\end{lem}
\begin{proof} If $\varphi$ is not $G$-invariant, then we may find $x \in \cM$ such that $g\phi(x)$ is not constant as $g$ varies in $G$. So the restriction of $\varphi$ to the von Neumann subalgebra $\cM_0$ generated by $G \cdot x$ is not $G$-invariant. We claim that $\cM_0$ has separable predual. Once this is proven, the theorem applied to $\cM_0$ will imply the conclusion for $\cM$.

We thus have to represent $\cM_0$ faithfully on a separable Hilbert space. Denote by $S \subset G$ a countable dense subset. Then the $\ast$-invariant $\Q$-algebra generated by $S \cdot x$ is a countable weakly dense subset of $\cM_0$. Therefore the GNS Hilbert space associated with $\varphi|_{\cM_0}$ is separable. Since moreover $\varphi$ is faithful on $\cM$ (hence on $\cM_0$), $\cM_0$ is faithfully represented on this separable GNS Hilbert space. This is the desired conclusion.
\end{proof}

Let us introduce some terminology. Let $(\mathcal M, \varphi)$ be any ergodic $(G, \mu)$-von Neumann algebra with separable predual and denote by $\sigma : G \curvearrowright \mathcal M$ the corresponding continuous action. Recall that $\varphi \in \mathcal M_\ast$ is a {\em faithful} normal state (see Lemma \ref{lem:support}). Therefore the corresponding GNS representation $\pi_\varphi$ is faithful, and we may thus identify $\pi_\varphi(\mathcal M)$ with $\mathcal M$ and assume that $\pi_\varphi = \id$. Choose a globally $G$-invariant ultraweakly dense separable unital $\rC^*$-subalgebra $\mathcal A \subset \mathcal M$ such that the action $G \curvearrowright \mathcal A$ is norm continuous (see the proof of \cite[Proposition XIII.1.2]{Ta03b}). Let $\psi \in \mathcal S(\mathcal A)$ be the unique $P$-invariant state corresponding to $\varphi |_{\mathcal A}$ so that $\varphi|_{\mathcal A} = \int_{G/P} \psi \circ \sigma_{g}^{-1} \, {\rm d}\nu_P(gP)$ (see Theorem \ref{thm:representation} and Corollary \ref{exmp:C*2}). 

\begin{lem} 
We may assume that $G$ has trivial center. 
\end{lem}
\begin{proof} Denote by $Z = Z(G)$ the center of $G$, which is finite, and set $G_0 = G/Z$, which is a connected semisimple Lie group satisfying the same assumptions as $G$ and which moreover has trivial center. Observe that $Z < K$ and $Z < P$. Set $K_0 = K/Z$ and $P_0 = P/Z$. Denote by $\mu' \in \Prob(G_0)$ the push-forward measure of $\mu$ under the quotient map $G \to G_0$. Then $\mu'$ is $K_0$-invariant and admissible. Denote by $\mathcal M_0 = \mathcal M^Z$ the fixed-point von Neumann algebra under the action of $Z$ and set $\varphi_0 = \varphi|_{\mathcal M_0}$. Then $(\mathcal M_0, \varphi_0)$ is an ergodic $(G_0, \mu')$-von Neumann algebra. Assume that Theorem \ref{thm:NZ} holds for this action.
Before deducing the theorem for $\cM$, let us record two facts.

{\bf Fact 1.} $\varphi$ is $Z$-invariant on $\cM$. 

Since $\varphi$ is normal, we only need to check that it is $Z$-invariant on the dense subalgebra $\cA$.
Since $Z \subset P$, $\psi$ is $Z$-invariant and we find, for every $z \in Z$, $a \in \mathcal A$,
\[(\varphi \circ \sigma_z)(a) = \int_{G/P} (\psi\circ \sigma_g^{-1})(\sigma_z(a)) \, {\rm d}\nu_P(gP)=  \int_{G/P} (\psi\circ \sigma_z)(\sigma_g^{-1}(a))  \, {\rm d}\nu_P(gP) = \varphi(a).\]

{\bf Fact 2.} $\varphi_0$ is $G_0$-invariant if and only if $\varphi$ is $G$-invariant.

The if direction is trivial.  To prove the only if direction, we define a map $\rE: \cM \to \cM_0$, by the formula $\rE(x) = \frac{1}{\vert Z \vert} \sum_{z \in Z} \sigma_z(x)$, for all $x \in \cM$.
This map is $G$-equivariant in the sense that $\rE \circ \sigma_g = \sigma_g \circ \rE$, for all $g \in G$. Moreover, Fact 1 implies that $\varphi \circ \rE = \varphi$. Assume now that $\varphi_0$ is $G_0$-invariant, i.e.\ that $\varphi$ is $G$-invariant on $\cM_0$. Then for all $x \in \cM$, $g \in G$, we have
\[\varphi(\sigma_g(x)) = \varphi \circ \rE(\sigma_g(x)) = \varphi(\sigma_g(\rE(x))) = \varphi(\rE(x)) = \varphi(x).\]

We can now deduce the conclusion for $\cM$. Assume that $\varphi$ is not $G$-invariant. Then, $\varphi_0$ is not $G_0$-invariant, and there exist a proper parabolic subgroup $P_0 \subset Q_0 \subsetneq G_0$ and a $G_0$-equivariant normal unital $\ast$-embedding $\Theta_0 : \rL^\infty(G_0/Q_0, \nu_{Q_0}) \to \mathcal M_0$ such that $\varphi_0 \circ \Theta_0 = \nu_{Q_0}$. Letting $Q = \pi^{-1}(Q_0)$, we have an identification $\rL^\infty(G/Q, \nu_Q) \simeq \rL^\infty(G_0/Q_0, \nu_{Q_0})$ as $G$-von Neumann algebras. We may thus compose $\Theta_0$ with this identification and get a $G$-equivariant normal unital $\ast$-embedding $\Theta : \rL^\infty(G/Q, \nu_Q) \to \mathcal M_0 \subset \mathcal M$ such that $\varphi \circ \Theta = \nu_{Q}$. 
\end{proof}

Let $(\pi_\psi, H_\psi, \xi_\psi)$ be the GNS triple associated with $(\mathcal A, \psi)$ and set $\mathcal N = \pi_\psi(\mathcal A)\dpr$. We also denote by $\psi$ the normal state $\langle \, \cdot \, \xi_\psi, \xi_\psi\rangle$ on $\mathcal N$. By definition, we have $\psi(\pi_\psi(a)) = \langle \pi_\psi(a)\xi_\psi, \xi_\psi\rangle= \psi(a)$ for every $a \in \mathcal A$. Since the action $P \curvearrowright \mathcal A$ is $\psi$-preserving, it extends to a continuous action $\sigma^{\mathcal N} : P \curvearrowright \mathcal N$ such that  $\sigma_g^{\mathcal N}(\pi_\psi(a)) = \pi_\psi(\sigma_g(a))$, for all $g \in P$, $a \in \mathcal A$ (see \cite[Exercice I.10.7]{Ta02}). Denote by $q \in \mathcal N$ the support projection of $\psi \in \mathcal N_\ast$. Recall that $q$ is the orthogonal projection of $H_\psi$ onto the closure of $\mathcal N'\xi_\psi$.  Since $\sigma^{\mathcal N} : P \curvearrowright \mathcal N$ is $\psi$-preserving, we have $q \in \mathcal N^P$. We point out that the action $\sigma^{\mathcal N} : P \curvearrowright \mathcal N$ need not be ergodic and $q$ need not be equal to~$1$. 

Our first task will be to embed $G$-equivariantly $\cM$ inside the induced von Neumann algebra of the action $P \actson \cN$. Before doing so, let us give some concrete facts on this action, as in \cite[Section 7]{NZ00}. We will use the notation from Section \ref{subsection:G/P}. Using \cite[Lemma IV.2.2]{Ma91}, the product map $\overline V \times P \to G : (\overline v, p) \mapsto \overline vp$ is a homeomorphism onto its image $ \overline V  P$ which is open and conull in $G$. 
As explained in \cite[IV.2.6]{Ma91}, the restriction of the quotient map $G \to G/P : g \mapsto gP$ to $\overline V$ gives a measure space isomorphism $(\overline V,\nu_{\overline V}) \to (G/P,\nu_P) : \overline v \mapsto \overline v P$, whose inverse is denoted by $\tau$. Observe that the Borel probability measure $\nu_{\overline V} \in \Prob(\overline V)$ is in the same class as the Haar measure $m_{\overline V}$. Modifying $\tau $ on a set of measure $0$, we can ensure that $\tau : G/P \to G$ is a measurable section to the quotient map $G \to G/P$ such that $\tau(\overline v P) = \overline v$, for all $\overline v \in \overline V$. As we saw in Subsection \ref{subsection:vN}, we may identify the induced action $G \actson \Ind_P^G(\mathcal N)$ with the continuous $G$-action $\tsigma: G \curvearrowright  \rL^\infty(G/P) \ovt \mathcal N$ given by the formula
\[\tsigma_g(F)(w) = \sigma^\cN_{c_\tau(g, g^{-1} w)}(F(g^{-1}w)), \text{ for all } F \in \rL^\infty(G/P) \ovt \mathcal N, g \in G, w \in G/P,\]
where $c_\tau : G \times G/P \to P$ is the measurable $1$-cocycle associated with $\tau$. Furthermore, using our measure space identification, we have a von Neumann algebra isomorphism
\[\rL^\infty(G/P) \ovt \mathcal N \cong \rL^\infty(\overline V) \ovt \cN.\]
The induced $G$-action may then be transported to a $G$-action on the right hand side von Neumann algebra. We can moreover obtain a very concrete formula for the action of $\overline P$. Indeed, observe that
\[g\overline v P =  \begin{cases} (g\overline{v})P \in \overline V P & \text{ for all } g \in \overline{V}, \overline{v} \in \overline V\\ (g \overline vg^{-1})P \in \overline V P & \text{ for all } g \in S, \overline{v} \in \overline V.\end{cases}\]
This computation was used by Margulis in his proof of the normal subgroup theorem. It implies formulae for the cocycle $c_\tau$: for every $\overline{v} \in \overline V$, $c_\tau(g, \overline v) = e$ if $g \in \overline V$ and $c_\tau(g, \overline v) = g$ if $g \in S$. Consequently,
\[ (\tsigma_g(F))(\bar v) = \begin{cases} F(g^{-1}\overline v) & \text{ if } g \in \overline V\\  \sigma^\cN_g (F(g^{-1} \overline vg)) & \text{ if } g \in S\end{cases},
\text{ for all } F \in \rL^\infty(\overline V) \ovt \cN, \overline v \in \overline V.\]

For any subset of simple roots $\theta \subset \Delta$, the semi-direct product decomposition $\overline V = \overline V_\theta \rtimes \overline U_\theta$ gives a von Neumann algebra isomorphism $\rL^\infty(\overline V) = \rL^\infty(\overline V_\theta \rtimes \overline U_\theta) = \rL^\infty(\overline V_\theta) \ovt \rL^\infty(\overline U_\theta)$. Recall that $S_\theta'$ commutes with $\overline U_\theta$. Therefore, the continuous actions $S'_\theta \curvearrowright \rL^\infty(\overline V) \ovt \mathcal N$ and $\oV_\theta \actson \rL^\infty(\overline V) \ovt \mathcal N$ are given by 
\begin{equation}\label{special action}
\sigma_s(F)(\overline v_\theta, \overline u_\theta) = \sigma_s^{\mathcal N}(F(s^{-1}\overline v_\theta s, \overline u_\theta)) \qquad  \text{ and } \qquad \sigma_g(F)(\overline v_\theta, \overline u_\theta) = F(g^{-1}\overline v_\theta, \overline u_\theta),
\end{equation}
for all $F \in \rL^\infty(\overline V_\theta \rtimes \overline U_\theta) \ovt \cN$, $s \in S_\theta'$, $g \in \oV_\theta$, $\ov_\theta \in \overline V_\theta$, $\ou_\theta \in \overline U_\theta$.

We will freely use this isomorphism $\rL^\infty(G/P) \ovt \mathcal N \cong \rL^\infty(\overline V) \ovt \cN$ to switch between these two points of view on the induced action, depending whether we want to emphasize the general $G$-action or the explicit formulae above. Given the $P$-invariant state $\psi$ on $\cN$, we may consider the associated canonical state $\nu_P \otimes \psi$ on $\rL^\infty(G/P) \ovt \cN$.

Consider the map $\iota : \mathcal A \to \rL^\infty(G/P) \ovt \cN$ defined by the formula $\iota(a)(w) = \pi_{\psi}(\sigma_{\tau(w)}^{-1}(a))$, for all $a \in \mathcal A$, $w \in G/P$. 
Under the identification $\rL^\infty(G/P) \ovt \mathcal N = \rL^\infty(\overline V) \ovt \mathcal N$, the mapping $\iota : \mathcal A \to \rL^\infty(\overline V) \ovt \mathcal N$ is given by the formula $\iota(a)(\overline v) = \pi_{\psi}(\sigma_{\overline v}^{-1}(a))$, for all $a \in \mathcal A$, $\overline v \in \overline V$.

\begin{lem}\label{lem:iota}
The map $\iota$ extends to a well-defined $G$-equivariant normal unital $\ast$-embedding $\iota : \mathcal M \to \rL^\infty(G/P) \ovt \cN$ such that $(\nu_P \otimes \psi) \circ \iota = \varphi$. 
\end{lem}
\begin{proof}
The proof of the claim is analogous to the one of Lemma \ref{GNS identification}. First, observe that for every $a \in \mathcal A$, $g \in G$ and $w \in G/P$, we have 
\begin{align*}
\iota(\sigma_g(a))(w) & = \pi_\psi(\sigma_{\tau(w)}^{-1}(\sigma_g(a))) \\
& = \pi_\psi(\sigma_{\tau(w)^{-1}g\tau(g^{-1}w)\tau(g^{-1}w)^{-1}}(a))\\
& = \sigma^\cN_{c_\tau(g,g^{-1}w)}\pi_\psi(\sigma_{\tau(g^{-1}w)^{-1}}(a)) = \tsigma_g(\iota(a))(w).
\end{align*}
Therefore $\iota(\sigma_g(a)) = \tsigma_g(\iota(a))$. Since $\psi \in \mathcal N_\ast$ is $P$-invariant, the bounded continuous map $G/P \to \C : gP \to \psi(\sigma_g^{-1}(a))$ is well-defined and we have
\[(\nu_{P} \otimes \psi)(\iota(a)) = \int_{G/P} \psi(\sigma_{\tau(w)}^{-1}(a) ) \, {\rm d}\nu_{P}(w) =\int_{G/P} \psi(\sigma_{g}^{-1}(a) ) \, {\rm d}\nu_{P}(gP) = \varphi(a), \text{ for all } a \in \cA.\]
Thus, once we proved that $\iota : \mathcal M \to \Ind_P^G(\mathcal N)$ extends to a normal  unital $\ast$-embedding, we will necessarily have that $\iota$ is $G$-equivariant and $(\nu_P \otimes \psi) \circ \iota = \varphi$.

Set $H = \rL^2(G/P, \nu_P) \otimes H_\psi$ and $\xi = 1_{G/P} \otimes \xi_\psi \in H$. Denote by $p \in \iota(\mathcal A)' \cap \mathbf B(H)$ the orthogonal projection onto the closed subspace $H_0 = \overline{\iota(\mathcal A)\xi}$. We identify $\mathbf B(H_0) = p\mathbf B(H)p$. Observe that $\xi$ is a $\iota(\mathcal A)$-cyclic vector in $H_0$ which implements the state $\varphi$ on $\mathcal A$. Thus, by uniqueness of the GNS representation, the unitary representation $\mathcal A \to \mathbf B(H_0) : a \mapsto \iota(a)p$ is unitarily conjugate to $\pi_\varphi = \id$. In particular, it indeed extends to a normal unital $\ast$-isomorphism $\mathcal M \to \iota(\mathcal A)\dpr p : a \mapsto \iota(a)p$. We are left to check that the normal unital $\ast$-homomorphism $\iota(\mathcal A)\dpr \to \iota(\mathcal A)\dpr p: f \mapsto fp$ is injective. Let $f \in \iota(\mathcal A)\dpr$ such that $fp = 0$. For every $a \in \mathcal A$, we have $f \iota(a)\xi = 0$. Regarding $f \in \rL^\infty(G/P, \mathcal N)$, for every $a \in \mathcal A$ and almost every $w \in G/P$, we have $f(w) \pi_\psi(\sigma_{\tau(w)}^{-1}(a))\xi_\psi = 0$. Since $\mathcal A$ is separable, this implies that for  almost every $w \in G/P$ and every $a \in \mathcal A$, we have $f(w) \pi_\psi(\sigma_{\tau(w)}^{-1}(a))\xi_\psi = 0$. Since $\xi_\psi$ is $\pi_\psi(\mathcal A)$-cyclic, we conclude that $f(w) = 0$ for almost every $w \in G/P$. This finally shows that $f = 0$.
\end{proof}

Lemma \ref{lem:iota} allows us to regard $\mathcal M \subset \rL^\infty(G/P) \ovt \cN$ as a $\widetilde \sigma(G)$-invariant von Neumann subalgebra on which the action $\sigma : G \curvearrowright \mathcal M$ coincides with the action $\widetilde \sigma : G \curvearrowright \rL^\infty(G/P) \ovt \cN$. From now on, we will equally use the letter $\sigma$ to denote any of these actions.

In Lemmas \ref{lem:NZ1} and \ref{lem:NZ2}, we fix $\theta \subset \Delta$ such that $\theta \neq \Delta$ and $s \in S_\theta'$. Lemma \ref{lem:NZ1} is a generalization of \cite[Proposition 7.1]{NZ00}. It will allow us to reduce to the case where  some elements of $P$ act trivially on $\cN$. Denote by $m_{\overline U_\theta}$ a Haar measure on $\overline U_\theta$.

\begin{lem}\label{lem:NZ1}
Let $a \in \mathcal A \subset \mathcal M$. For every $n \in \N$, define $a_n = \frac{1}{n + 1} \sum_{k = 0}^n \sigma_{s}^k(a) \in \mathcal A$. Then the sequence $(a_n)_{n \in \N}$ converges strongly in $\mathcal M$. Set $a_\infty = \lim_n a_n $ to be its strong limit in $\mathcal M$. The following assertions hold:
\begin{itemize}
\item We have $\iota(a_\infty) \in \C 1_{\overline V_\theta} \ovt \rL^\infty(\overline U_\theta) \ovt \mathcal N^s$. We may regard $\iota(a_\infty) \in \rL^\infty(\overline U_\theta, \mathcal N^s)$.
\item There exists a subsequence $(n_k)_{k \in \N}$ such that for $m_{\overline U_\theta}$-almost every $\overline u_\theta \in \overline U_\theta$, the sequence $(\iota(a_{n_k})(\overline e, \overline u_\theta))_{k \in \N}$ converges strongly to $\iota(a_\infty)(\overline u_\theta)$ in $\mathcal N$.
\end{itemize}
\end{lem}

\begin{proof}
The proof follows the same strategy as the one of \cite[Proposition 7.1]{NZ00}. However, it involves extra technicalities because the normal state $\psi \in \mathcal N_\ast$ need not be faithful. We start by proving that the sequence $(a_n)_{n \in \N}$ converges strongly in $\mathcal M$. Since $(a_n)_{n \in \N}$ is uniformly bounded and $\varphi \in \mathcal M_\ast$ is faithful, it suffices to show that the sequence $(a_n)_{n \in \N}$ is $\|\cdot\|_\varphi$-Cauchy. Let $(\overline v_\theta, \overline u_\theta) \in \overline V = \overline V_\theta \rtimes \overline U_\theta$. Let $\varepsilon > 0$. Since $a \in \mathcal A$, since the action $G \curvearrowright \mathcal A$ is norm continuous and since $s^{-k} \overline v_\theta s^k \to e$ in $\overline V$ as $k \to \infty$, there exists $k_0 = k_0(\overline v_\theta, \overline u_\theta) \in \N$ such that for every $k \geq k_0$, we have $\|\sigma_{s^{-k}\overline v_\theta^{-1} s^k}(a) - a\| \leq \varepsilon$. For all $k \geq k_0$, we have
\begin{align*}
\|\iota(\sigma_{s}^k(a))(\overline v_\theta, \overline u_\theta) - \iota(\sigma_{s}^k(a))(\bar e,\overline u_\theta)\| &= \|\pi_\psi(\sigma_{\overline u_\theta}^{-1}(\sigma_{\overline v_\theta}^{-1} (\sigma_s^k(a)))) - \pi_\psi(\sigma_{\overline u_\theta}^{-1} (\sigma_s^k(a)))\| \\
& \leq \|\sigma_{s^{-k}\overline v_\theta^{-1} s^k}(a) - a\|\\
& \leq \eps.
\end{align*}
Since $\eps > 0$ was arbitrary, we deduce that
\begin{equation}\label{eq:norm}
\lim_n \|\iota(a_n)(\overline v_\theta, \overline u_\theta) - \iota(a_n)(\overline e, \overline u_\theta)\| = 0.
\end{equation}
Since the action $\sigma : P \curvearrowright \mathcal N$ is $\psi$-preserving, we may define the strongly continuous Koopman unitary representation $\kappa : P \to \mathcal U(H_\psi)$ by the formula $\kappa(g)( b \xi_\psi) = \sigma_g(b)\xi_\psi$, for all $g \in P$ and $b \in \mathcal N$. Denote by $P_{\kappa(s)} : H_\psi \to H_\psi$ the orthogonal projection onto the closed subspace of $\kappa(s)$-invariant vectors. By von Neumann's ergodic theorem, the sequence $(\frac{1}{n+1} \sum_{k = 0}^n\kappa(s)^k)_{n \in \N}$ converges strongly to $P_{\kappa(s)}$. Since $s$ commutes with $\overline u_\theta$, this implies that $\lim_n \|\iota(a_n)(\overline e, \overline u_\theta) \xi_\psi - P_{\kappa(s)}(\pi_\psi(\sigma_{\overline u_\theta}^{-1}(a)) \xi_\psi)\| = 0$. Combining this with \eqref{eq:norm}, we obtain that 
\begin{equation}\label{pointwise cv}\lim_n \|\iota(a_n)(\overline v_\theta, \overline u_\theta)  \xi_\psi- P_{\kappa(s)}(\pi_\psi(\sigma_{\overline u_\theta}^{-1}(a)) \xi_\psi)\| = 0.\end{equation}
In particular, we obtain that the sequence $(\iota(a_n)(\overline v_\theta, \overline u_\theta))_{n \in \N}$ is $\|\cdot\|_\psi$-Cauchy in $\mathcal N$. Since $(\nu_{\overline V} \otimes \psi) \circ \iota = \varphi$, we have, for all $m, n \in \N$,
\[
\|a_{m} - a_{n}\|_\varphi^2 = \int_{\overline V} \|\iota(a_{m})(\overline v_\theta, \overline u_\theta) - \iota(a_{n})(\overline v_\theta, \overline u_\theta)\|_\psi^2 \, {\rm d}\nu_{\overline V}(\overline v_\theta, \overline u_\theta).
\]
Applying Lebesgue's dominated convergence theorem, we obtain that the sequence $(a_n)_{n \in \N}$ is $\|\cdot\|_\varphi$-Cauchy and thus strongly convergent in $\mathcal M$.

Set $a_\infty = \lim_n a_n \in \mathcal M$ to be the strong limit of the sequence $(a_n)_{n \in \N}$ in $\mathcal M$. We next show that $\iota(a_\infty) \in \C 1_{\oV_\theta} \ovt \rL^\infty(\overline U_\theta) \ovt \mathcal N^s$. To that aim, let us first check that $a_\infty$ is $\oV_\theta$-invariant. Observe that the action of $\oV_\theta$ on $\rL^\infty(\oV) \ovt \cN \cong \rL^\infty(\oV_\theta) \ovt \rL^\infty(\oU_\theta) \ovt \cN$ is just the action on the first tensor factor $\rL^\infty(\oV_\theta)$ by left translation. Taking $\ov \in \oV_\theta$, we have, for all $n \in \N$,
\[\Vert \sigma_\ov(a_n) - a_n \Vert_\varphi = \int_\oV \Vert \iota(a_{n})(\ov^{-1}\overline v_\theta, \overline u_\theta) - \iota(a_{n})(\overline v_\theta, \overline u_\theta)\Vert_\psi^2 \, {\rm d}\nu_{\overline V}(\overline v_\theta, \overline u_\theta).\]
Equation \eqref{pointwise cv} and Lebesgue's dominated convergence theorem imply that this quantity converges to $0$ as $n \to \infty$. Therefore, $a_\infty$ is $\oV_\theta$-invariant, which implies that $\iota(a_\infty) \in \C 1_{\oV_\theta} \ovt \rL^\infty(\overline U_\theta) \ovt \mathcal N$. Moreover, since $\Vert a_n - \sigma_s(a_n)\Vert$ tends to $0$ as $n \to \infty$, it follows that $a_\infty$ is $s$-invariant. In view of the formula for the induced action of $s \in S_\theta'$, this implies that $\iota(a_\infty) \in  \C 1_{\oV_\theta} \ovt \rL^\infty(\overline U_\theta) \ovt \mathcal N^s$. Regard $\iota(a_\infty) \in \rL^\infty(\overline U_\theta, \mathcal N^s)$.

We now check the last statement. Since $\iota(a_n)$ converges strongly to $\iota(a_\infty)$ as $n \to \infty$, there exists a subsequence $(a_{n_k})_{k \in \N}$ such that $\iota(a_{n_k})(\ov_\theta,\ou_\theta)$ converges strongly to $\iota(a_\infty)(\ov_\theta,\ou_\theta) = \iota(a_\infty)(\ou_\theta)$ for almost every $(\ov_\theta,\ou_\theta)\in \oV$. In view of \eqref{eq:norm}, we deduce that  $\iota(a_{n_k})(\overline e,\ou_\theta)$ converges strongly to $\iota(a_\infty)(\ou_\theta)$ for almost every $(\ov_\theta,\ou_\theta)\in \oV$. In particular, Fubini's theorem implies that there exists $\ov_\theta \in \oV_\theta$, such that $\iota(a_{n_k})(\overline e,\ou_\theta)$ converges strongly to $\iota(a_\infty)(\ou_\theta)$ for almost every $\ou_\theta \in \oU_\theta$.
\end{proof}

Set $W_{\theta, s} = s^\Z \ltimes V_\theta \subset P$. It follows from the Levi decomposition $P_\theta = Z_G(S_\theta) \ltimes V_\theta$ that $W_{\theta,s}$ is a normal subgroup of $P_\theta$. Since $W_{\theta, s} \subset P \subset P_\theta$, it follows that $W_{\theta, s}$ is a normal subgroup of $P$.

We have the following inclusions of fixed point von Neumann subalgebras: 
$$\mathcal N^P \subset \mathcal N^{W_{\theta, s}} \subset \mathcal N^s \subset \mathcal N.$$ 
In order to properly use the previous lemma, we would like to show that $\cN^s$ is globally $P$-invariant and study the action of $G$ on $\rL^\infty(G/P) \ovt \cN^s$. In the commutative setting, Mautner's phenomenon implies that $\cN^s = \cN^{W_{\theta,s}}$, which is $P$-invariant (see \cite[Section 4]{NZ00}). This is because in the case when $\cN$ is commutative, the state $\psi$ is faithful on $\cN$, and so Mautner's phenomenon, which is based on purely Hilbert space considerations, can be applied. Unfortunately, in the general noncommutative case, the state $\psi$ need not be faithful on $\cN$, and Mautner's phenomenon fails. The equality $\cN^s = \cN^{W_{\theta,s}}$ does not hold. However, this equality will hold under taking a corner with respect to a  suitable projection. This is the point of the following notation and lemma.

Denote by $q_{\theta, s} = \bigvee_{u \in \mathcal U(\mathcal N^{W_{\theta, s}})} uqu^* \in \mathcal Z( \mathcal N^{W_{\theta, s}})$ the central support in $\mathcal N^{W_{\theta, s}}$ of the projection $q \in \mathcal N^P \subset \mathcal N^{W_{\theta, s}}$. Here, the notation $\bigvee$ designates the join operation in the lattice of projections, i.e.\ the smallest projection dominating each element in the family. Since $W_{\theta, s}$ is a normal subgroup of $P$, for every $g \in P$, we have $\sigma_g (\mathcal N^{W_{\theta, s}}) = \mathcal N^{W_{\theta, s}}$ and so
$$\sigma_g(q_{\theta, s}) = \bigvee_{u \in \mathcal U(\mathcal N^{W_{\theta, s}})} \sigma_g(u) \sigma_g(q) \sigma_g(u^*) = \bigvee_{u \in \mathcal U(\mathcal N^{W_{\theta, s}})} \sigma_g(u) q \sigma_g(u^*) = \bigvee_{u \in \mathcal U(\mathcal N^{W_{\theta, s}})} u q u^* = q_{\theta, s}.$$
Therefore, we have $q \leq q_{\theta, s}$ and $q_{\theta, s} \in \mathcal N^P \cap \mathcal Z(\mathcal N^{W_{\theta, s}})$. Observe that $q_{\theta, s}$ is the orthogonal projection from $H_\psi$ onto the closure of $(\mathcal N^{W_{\theta, s}} \vee \mathcal N')\xi_\psi$. The following lemma follows from Mautner's phenomenon.

\begin{lem}\label{lem:NZ2}
For every $y \in \mathcal N^s$, we have $y q_{\theta, s}  = q_{\theta, s} y$ and $y q_{\theta, s} \in \mathcal N^{W_{\theta, s}}$.
\end{lem}
\begin{proof}
Let $y \in \mathcal N^s$ and $b \in \mathcal N^{W_{\theta, s}}$. We still denote by $\kappa : P \to \mathcal U(H_\psi)$ the strongly continuous Koopman representation associated with the $\psi$-preserving action $\sigma : P \curvearrowright \mathcal N$. We have $y b \in \mathcal N^s$ and so $\kappa(s)(y b \xi_\psi) = y b \xi_\psi$. By Mautner's phenomenon applied to $\kappa$ (see \cite[Lemma II.3.2]{Ma91}), we have $\kappa(g)(y b \xi_\psi) = y b \xi_\psi$, for every $g \in V_\theta$. Therefore, we have $\kappa(g)(y b \xi_\psi) = y b \xi_\psi$ for every $g \in W_{\theta, s} = s^\Z \ltimes V_\theta$. This implies that for every $g \in W_{\theta, s}$ and every $b \in \mathcal N^{W_{\theta, s}}$, we have 
$$\sigma_g(y) b \xi_\psi = \sigma_g(y b)  \xi_\psi  = \kappa(g)(y b \xi_\psi) = y b \xi_\psi.$$
Hence, for every $g \in W_{\theta, s}$, every $b \in \mathcal N^{W_{\theta, s}}$ and every $c \in \mathcal N'$, we have 
$$\left( \sigma_g(y) -y \right) bc \xi_\psi = 0.$$
This shows that for every $g \in W_{\theta, s}$, $\left( \sigma_g(y) - y\right) q_{\theta, s} = 0$ and so $\sigma_g(yq_{\theta, s}) = \sigma_g(y)q_{\theta, s} = yq_{\theta, s}$. Thus, $y q_{\theta, s} \in \mathcal N^{W_{\theta, s}}$. Since $q_{\theta, s} \in \mathcal Z(\mathcal N^{W_{\theta, s}})$, we have $y q_{\theta, s} =  q_{\theta, s}y q_{\theta, s}$. Applying this equality to $y^* \in \mathcal N^s$ shows that $q_{\theta, s} y = (y^*q_{\theta, s})^* = (q_{\theta, s} y^*q_{\theta, s})^* = q_{\theta,s}yq_{\theta,s} = yq_{\theta,s}$.
\end{proof}

A combination of Lemmas \ref{lem:NZ1} and \ref{lem:NZ2} shows that for every $a \in \mathcal A$, we have that
$$\iota(a_\infty)(1_{G/P} \otimes q_{\theta, s}) = (1_{G/P} \otimes q_{\theta, s}) \iota(a_\infty) \quad \text{and} \quad \iota(a_\infty)(1_{G/P} \otimes q_{\theta, s}) \in \rL^\infty(G/P) \ovt \mathcal N^{W_{\theta, s}}q_{\theta, s}.$$

We now have all the required tools to start the proof of Theorem \ref{thm:NZ}.

\begin{proof}[Proof of Theorem \ref{thm:NZ}]
We assume that $\varphi$ is not $G$-invariant and we will show that there exist a proper parabolic subgroup $P \subset Q \subsetneq G$ and a $G$-equivariant normal unital $\ast$-embedding $\Theta : \rL^\infty(G/Q, \nu_Q) \to \mathcal M$ such that $ \varphi \circ \Theta = \nu_Q$. Since $\varphi$ is not $G$-invariant and since $\varphi \in \mathcal M_\ast$, $\varphi|_{\mathcal A}$ is not $G$-invariant either. Recall that $\varphi |_{\mathcal A} = \int_{G/P} \psi \circ \sigma_g^{-1} \, {\rm d}\nu_P(gP)$. This implies that $\psi$ is not $G$-invariant. Since the real rank of $G$ is at least two, the proof of \cite[Proposition I.1.2.2]{Ma91} shows that the subgroups $\overline U_\theta$ with $\emptyset \subsetneq \theta \subsetneq \Delta$  generate the subgroup $\overline V$. By \cite[Proposition I.1.2.1]{Ma91}, the subgroups $P$ and $\overline V$ generate $G$. Since $\psi$ is $P$-invariant but not $G$-invariant, the above reasoning shows that there exists $\emptyset \subsetneq \theta \subsetneq \Delta$ such that $\psi$ is not $\overline U_\theta$-invariant. We fix such a subset $\theta \subset \Delta$ and we choose  $s \in S'_\theta$.

The strategy of the {\bf first part} of proof is strongly inspired by the techniques developed by Nevo--Zimmer in \cite{NZ00} (most notably the proofs of \cite[Theorem 1]{NZ00} and \cite[Proposition 10.1]{NZ00}). 

\begin{claim}\label{claim:nonconstant}
For every $a \in \cA$, for $m_{\overline U_\theta}$-almost every $\ou_\theta \in \oU_\theta$, we have $\psi(\iota(a_\infty)(\overline u_\theta)) = \psi(\sigma_{\ou_\theta}^{-1}(a))$. In particular, there exists $a \in \mathcal A$ such that the bounded measurable function $\overline U_\theta \to \C : \overline u_\theta \mapsto \psi(\iota(a_\infty)(\overline u_\theta) )$ is non-constant in $\rL^\infty(\overline U_\theta)$.
\end{claim}

\begin{proof}[Proof of Claim \ref{claim:nonconstant}]
By Lemma \ref{lem:NZ1}, there exists a subsequence $(n_k)_{k \in \N}$ such that for $m_{\overline U_\theta}$-almost every $\overline u_\theta \in \overline U_\theta$, $\iota(a_{n_k})(\overline e, \overline u_\theta) \to \iota(a_\infty)(\overline u_\theta)$ strongly in $\mathcal N$. Since the action $P \curvearrowright \mathcal N$ is $\psi$-preserving and since $s$ commutes with $\overline U_\theta$, we have $\psi(\iota(a_{n_k})(\overline e, \overline u_\theta)) = \psi(\iota(a)(\overline e, \overline u_\theta)) = \psi(\sigma_{\ou_\theta}^{-1}(a))$ for every $k \in \N$ and every $\overline u_\theta \in \overline U_\theta$. This implies that $\psi(\iota(a_\infty)(\overline u_\theta)) = \lim_k \psi(\iota(a_{n_k})(\overline e, \overline u_\theta)) = \psi(\sigma_{\ou_\theta}^{-1}(a))$ for $m_{\overline U_\theta}$-almost every $\overline u_\theta \in \overline U_\theta$.

To prove the second part, since $\psi$ is not $\overline U_\theta$-invariant, we may choose $a \in \cA$ such that the (continuous) function $ \oU_\theta \to \C : \ou_\theta \mapsto \psi(\sigma_{\ou_\theta}^{-1}(a))$ is non-constant. Since this function is continuous on $\oU_\theta$, we deduce that it is not $m_{\overline U_\theta}$-almost everywhere constant on $\oU_\theta$, so the second part of the claim follows from the first.
\end{proof}

Set $q_0 = q_{\theta, s} \in \mathcal N^P \cap \mathcal Z(\mathcal N^{W_{\theta, s}})$ and $\mathcal N_0 = \mathcal N^{W_{\theta, s}}q_0$. Recall that $q \leq q_0$. Since $q_0 \in \mathcal N^P$, we have that $\mathcal N_0 \subset q_0 \mathcal N q_0$ is a globally $P$-invariant von Neumann subalgebra. Define
$$\mathcal M_0 = \{x \in \mathcal M \mid \iota(x)(1_{G/P} \otimes q_0) = (1_{G/P} \otimes q_0) \iota(x) \quad \text{and} \quad \iota(x)(1_{G/P} \otimes q_0) \in \rL^\infty(G/P) \ovt \cN_0 \}.$$
Since $\mathcal N_0 \subset q_0 \mathcal N q_{0}$ is globally $P$-invariant, since $(1 \otimes q_0) \in (\rL^\infty(G/P) \ovt \cN)^G$ and since $\iota$ is a $G$-equivariant normal unital $\ast$-embedding, it follows that $\mathcal M_0 \subset \mathcal M$ is a globally $G$-invariant von Neumann subalgebra. Define the $G$-equivariant normal unital $\ast$-homomorphism 
\[\iota_0 : \cM_0 \to  \rL^\infty(G/P) \ovt \cN_0 : x \mapsto \iota(x)(1_{G/P} \otimes q_0).\]
Since $q = \supp(\psi)$ and since $q \leq q_0$, we have $(\nu_P \otimes \psi)\circ \iota_0 = \varphi$ and so $\iota_0$ is indeed a $\ast$-embedding. Since the action $G \curvearrowright \mathcal M$ is ergodic, the action $G \curvearrowright \mathcal M_0$ is ergodic.

\begin{claim}\label{claim:nonpreserving}
The action $G \curvearrowright \mathcal M_0$ is not $\varphi$-preserving. In particular,  we have $\cM_0 \neq \C 1$.
\end{claim}

\begin{proof}[Proof of Claim \ref{claim:nonpreserving}]
By contradiction, assume that the action $G \curvearrowright \mathcal M_0$ is $\varphi$-preserving. Since $\iota_0$ is a $G$-equivariant normal unital $\ast$-embedding such that $(\nu_P \otimes \psi) \circ \iota_0 = \varphi$ and since the continuous $P$-action $\sigma : P \curvearrowright \mathcal N_0$ is $\psi$-preserving, for every $x \in \mathcal M_0$, the following quantity does not depend on $g \in G$:
\begin{align*}
\varphi(\sigma_g^{-1}(x)) &= \int_{G/P} \psi(\iota_0(\sigma_g^{-1}(x))(w)) \, {\rm d}\nu_P(w)\\
&= \int_{G/P} \psi(\sigma_{c_\tau(g^{-1}, gw)}(\iota_0(x)(gw))) \, {\rm d}\nu_P(w) \\
&= \int_{G/P} \psi(\iota_0(x)(gw)) \, {\rm d}\nu_P(w).
\end{align*}
Therefore, the bounded $\mu$-harmonic function $G \to \C : g \mapsto \int_{G/P} \psi(\iota_0(x)(gw)) \, {\rm d}\nu_P(w)$ is constant. Since $(G/P, \nu_P)$ is the $(G, \mu)$-Poisson boundary,  Theorem \ref{thm:harmonicity} implies that the bounded measurable function $G/P \to \C : w \mapsto \psi(\iota_0(x)(w))$ is $\nu_P$-almost everywhere constant.
 
Regarding $\rL^\infty(G/P, \nu_P) = \rL^\infty(\overline V, \nu_{\overline V})$, since $a_\infty \in \mathcal M_0$ and since $q \leq q_0$, we have $\psi(\iota(a_\infty)(\overline u_\theta) ) = \psi(\iota_0(a_\infty)(\overline u_\theta) )$ for $m_{\overline U_\theta}$-almost every $\overline u_\theta \in \overline U_\theta$ and so the bounded measurable function $\overline V \to \C : (\overline v_\theta, \overline u_\theta) \mapsto \psi(\iota(a_\infty)(\overline u_\theta) )$ is  $\nu_{\overline V}$-almost everywhere constant. This contradicts Claim \ref{claim:nonconstant}.
\end{proof}

At this point of the proof, we have constructed a globally $G$-invariant von Neumann subalgebra $\mathcal M_0 \subset \mathcal M$ such that the ergodic $(G, \mu)$-von Neumann algebra $(\mathcal M_0, \varphi)$ satisfies the following properties:
\begin{enumerate}
\item There exist $\emptyset \subsetneq \theta \subsetneq \Delta$, $s \in S_\theta'$, a von Neumann algebra $\mathcal N_0$ (with separable predual) endowed with a normal state $\psi$ and a continuous action $\sigma : P \curvearrowright \mathcal N_0$ that is $\psi$-preserving and such that $W_{\theta,s} \subset \ker(\sigma)$.
\item There exists a $G$-equivariant normal unital $\ast$-embedding $\iota_0 : \mathcal M_0 \to \rL^\infty(G/P) \ovt \cN_0$ such that $(\nu_P \otimes \psi) \circ \iota_0 = \varphi$.
\item The action $G \curvearrowright \mathcal M_0$ is not $\varphi$-preserving.
\end{enumerate}
As the action $G \curvearrowright \mathcal M_0$ may not be faithful, we denote by $N \lhd G$ its kernel and we consider the continuous action $G/N \curvearrowright \mathcal M_0$. We note that the quotient map $\pi_{G/N} : G \to G/N$ sends the $K$-invariant admissible Borel probability measure $\mu \in \Prob(G)$ to a $\pi_{G/N}(K)$-invariant admissible Borel probability measure $\overline \mu \in \Prob(G/N)$ and that the action $G/N \curvearrowright (\mathcal M_0, \varphi)$ is ergodic and $\overline \mu$-stationary. Moreover $G/N$ is a connected semisimple Lie group with trivial center, no nonzero compact factor, all of whose simple factors are of real rank at least two. Since the action $G/N \curvearrowright \mathcal M_0$ is not $\varphi$-preserving, we can repeat the first part of the proof of Theorem \ref{thm:NZ}. Since $G$ has only finitely many simple factors, we can continue a finite number of steps and we obtain that there exist a factor group $\pi_H : G \to H$, a $\pi_H(K)$-invariant admissible Borel probability measure $\mu_H \in \Prob(H)$, a globally $G$-invariant von Neumann subalgebra $\mathcal M_H \subset \mathcal M$ such that the continuous action $H \curvearrowright \mathcal M_H$ is {\em faithful}, ergodic, $\mu_H$-stationary and satisfies the aforementioned properties $(1)$, $(2)$, $(3)$. We will next show that there exist a proper parabolic subsgroup $F < H$ and a $H$-equivariant normal unital $\ast$-embedding $\Theta : \rL^\infty(H/F) \to \mathcal M_H$ such that $\varphi \circ \Theta = \nu_{H/F}$. Write $G = H \times N$ where $N$ is the kernel of the factor map $G \to H$. Then $Q = F \times N < G$ is a proper parabolic subgroup, $(G/Q, \nu_Q) \cong (H/F, \nu_{H/F})$ and $\Theta : \rL^\infty(G/Q) \to \mathcal M_H$ is a $G$-equivariant normal unital $\ast$-embedding such that $\varphi \circ \Theta = \nu_{Q}$. Thus, we may still denote by $G$ the factor group $H$ and by $\mathcal M_0 \subset \mathcal M$ the globally $G$-invariant von Neumann subalgebra. Moreover, the continuous action $G \curvearrowright \mathcal M_0$ is faithful, ergodic, $\mu$-stationary and does not preserve $\varphi$.

Note that the above iterative procedure is unnecessary if $G$ is simple. As we explain in Remark \ref{simple case} below, one can in fact deduce the semisimple case from the simple case thanks to an easy observation from \cite{BBHP20}.

The strategy of the {\bf second part} of the proof of Theorem \ref{thm:NZ} is new compared to the proof of \cite[Theorem 1]{NZ00}. We make use of various von Neumann algebraic techniques involving essential values into noncommutative algebras, disintegration theory and the slice mapping theorem of Ge--Kadison \cite{GK95} and its generalization by Str\u{a}til\u{a}--Zsid\'o \cite{SZ98}. We will show that the continuous $G$-action $G \curvearrowright \mathcal Z(\mathcal M_0)$ (which is ergodic and  $\mu$-stationary) does not preserve $\varphi$. We will then apply \cite[Theorem 1]{NZ00} to obtain the second item stated in Theorem \ref{thm:NZ}.

Recall that $\overline V = \overline V_\theta \rtimes \overline U_\theta$ so that $\rL^\infty(\overline V) = \rL^\infty(\overline V_\theta) \ovt \rL^\infty(\overline U_\theta)$. Choose a Borel probability measure $\nu_{\overline V_\theta} \in \Prob(\overline V_\theta)$ that is in the same class as the Haar measure $m_{\overline V_\theta}$. Put $\mathcal Q = \rL^\infty(\overline U_\theta) \ovt \mathcal N_0$ and regard 
$$ \rL^\infty(\overline V) \ovt \mathcal N_0 \cong \rL^\infty(\overline V_\theta) \ovt (\rL^\infty(\overline U_\theta) \ovt \mathcal N_0) \cong \rL^\infty(\overline V_\theta, \mathcal Q).$$ 
As usual, we endow $\mathcal Q$ with the strong operator topology. Let $f \in \rL^\infty(\overline V_\theta, \mathcal Q)$, which we regard as an element of $\rL^\infty(\overline V_\theta, \Ball_{\mathcal Q}(0, \|f\|_\infty))$. Denote by $F_f \subset \Ball_{\mathcal Q}(0, \|f\|_\infty)$ the essential range of $f$. Since $\mathcal Q$ has separable predual, $\Ball_{\mathcal Q}(0, \|y\|_\infty) \subset \mathcal Q$ is Polish with respect to the strong operator topology. Lemma \ref{lem:essential-range} implies that the essential range $F_f \subset \Ball_{\mathcal Q}(0, \|f\|_\infty)$ of $f$ is strongly closed in $\mathcal Q$ and coincides with the set of essential values of $f$. Moreover, we may view $f$ as an element of $\rL^\infty(\overline V_\theta, F_f)$. 

Regard $\iota_0(\mathcal M_0) \subset \rL^\infty(\overline V_\theta, \mathcal Q)$.  Choose a strongly dense countable subset $\{x_n \; | \; n \in \N\}$ of $\mathcal M_0$. Define the von Neumann subalgebra $\mathcal Q_0 \subset \mathcal Q$ by
$$\mathcal Q_0= \bigvee \{ F_{\iota_0(x_n)} \mid n \in \N \}.$$

\begin{claim}\label{claim:inclusion}
We have $\C1_{\overline V_\theta} \ovt \mathcal Q_0 \subset \iota_0(\mathcal M_0) \subset \rL^\infty(\overline V_\theta) \ovt \mathcal Q_0$.
\end{claim}

\begin{proof}[Proof of Claim \ref{claim:inclusion}]
By construction of $\cQ_0$, we have $\iota_0(x_n) \in \rL^\infty(\overline V_\theta, \cQ_0)$, for every $n \in \N$.  Since $\{x_n \; | \; n \in \N\}$ is strongly dense in $\mathcal M_0$ and since $\mathcal Q_0$ is a von Neumann algebra, this further implies that  $\iota_0(\mathcal M_0) \subset \rL^\infty(\overline V_\theta, \mathcal Q_0) = \rL^\infty(\overline V_\theta) \ovt \mathcal Q_0$. 

Choose a faithful state $\Psi \in (\mathcal Q_0)_\ast$. Let $n \in \N$ and set $f = \iota_0(x_n) \in F_{\iota_0(x_n)} \subset \mathcal Q_0$. Regard $f \in \rL^\infty(\overline V_\theta, F_{f})$ with $F_{f} \subset \Ball_{\mathcal Q_0}(0, \|f\|_\infty)$. Let $b \in F_f$ be any essential value of $f$.  For every $k \in \N$, define 
$$\overline B_k = \left\{ \overline v_\theta \in \overline V_\theta \mid \| f(\overline v_\theta) - b\|_\Psi < \frac{1}{k + 1}\right \}.$$
Since $b \in F_{f}$, we have $\nu_{\overline V_\theta}(\overline B_k) > 0$. For every $k \in \N$, applying \cite[Lemma IV.2.5(a)]{Ma91} to $\overline B_k$, there exist $\overline h_k \in \overline V_\theta$ and $n_k \in \N$ large enough so that $\lim_k \nu_{\overline V_\theta} (s^{n_k} \, \overline h_k \overline B_k \, s^{-n_k}) = 1$. For every $k \in \N$, set $g_k = s^{n_k} \overline h_k \in s^\Z \ltimes \overline V_\theta \subset \overline P$. We will show that $\sigma_{g_k}(f)$ converges strongly to $1_{\oV_\theta} \otimes b$ in $\mathcal Q_0$ as $k \to \infty$. This will clearly imply that $1_{\oV_\theta} \otimes b \in \iota_0(\cM_0)$ and the claim will follow.

In view of formula \eqref{special action} and since $s$ acts trivially on $\cN_0$, observe that 
\[\sigma_{g_k}(f)(\ov_\theta) = \sigma_{s}^{n_k}(f((\overline h_k)^{-1}s^{-n_k} \ov_\theta s^{n_k})) = f((\overline h_k)^{-1}s^{-n_k} \ov_\theta s^{n_k}), \text{ for all } \ov_\theta \in \oV_\theta.\]
For all $k \in \N$, set $\overline C_k = s^{n_k} \, \overline h_k \overline B_k \, s^{-n_k} \subset \overline V_\theta$. By assumption, we have $\lim_k \nu_{\overline V_\theta} (\overline C_k ) = 1$ and we may compute
\begin{align*}
\|\sigma_{g_k}(f) - 1_{\overline V_\theta} \otimes b\|_{\nu_{\overline V_\theta} \otimes \Psi}^2 &= \int_{\overline V_\theta} \|f((\overline h_k)^{-1}s^{-n_k} \ov_\theta s^{n_k}) - b\|_\Psi^2 \, {\rm d}\nu_{\overline V_\theta}(\overline v_\theta) \\
&= \int_{\overline C_k} \|f((\overline h_k)^{-1}s^{-n_k} \ov_\theta s^{n_k}) - b\|_\Psi^2 \, {\rm d}\nu_{\overline V_\theta}(\overline v_\theta)  \\
&\qquad + \int_{\overline V_\theta \setminus \overline C_k} \|f((\overline h_k)^{-1}s^{-n_k} \ov_\theta s^{n_k}) - b\|_\Psi^2 \, {\rm d}\nu_{\overline V_\theta}(\overline v_\theta) \\
&\leq \frac{1}{(k + 1)^2} \, \nu_{\overline V_\theta}(\overline C_k) + 2 \|f\|_\infty \, \nu_{\overline V_\theta}(\overline V_\theta \setminus \overline C_k). 
\end{align*}
This quantity converges to $0$ as $k \to \infty$. Since the sequence $(\sigma_{g_k}(f))_{k \in \N}$ is uniformly bounded, it follows that $\sigma_{g_k}(f) \to 1_{\overline V_\theta} \otimes b$ strongly as $k \to \infty$. 
\end{proof}

Claim \ref{claim:inclusion} implies that $\C1_{\overline V_\theta} \ovt \mathcal Z(\mathcal Q_0) \subset \mathcal Z(\iota_0(\mathcal M_0)) \subset \rL^\infty(\overline V_\theta) \ovt \mathcal Z(\mathcal Q_0)$.

\begin{claim}\label{claim:factor}
We have that $ \C1_{\overline V_\theta} \ovt \mathcal Z(\mathcal Q_0) \neq \mathcal Z(\iota_0(\mathcal M_0))$. 
\end{claim}

\begin{proof}[Proof of Claim \ref{claim:factor}]
The following approach was suggested to us by Narutaka Ozawa. While our original proof relies on Ge--Kadison's splitting theorem (see \cite[Theorem A]{GK95}), his argument relies on Str\u{a}til\u{a}--Zsid\'o's generalization of Ge--Kadison's result (see \cite[Theorem 4.2]{SZ98}). By contradiction, assume that $ \C1_{\overline V_\theta} \ovt \mathcal Z(\mathcal Q_0) = \mathcal Z(\iota_0(\mathcal M_0))$. Since 
$$\iota_0(\mathcal M_0) \cap (\rL^\infty(\overline V_\theta) \ovt \mathcal Z(\mathcal Q_0)) = \mathcal Z(\iota_0(\mathcal M_0)) = \C 1_{\overline V_\theta} \ovt \mathcal Z(\mathcal Q_0)$$ splits, \cite[Theorem 4.2]{SZ98} implies that $\iota_0(\mathcal M_0) = \C 1_{\overline V_\theta} \ovt \mathcal Q_0$ splits. Since $s$ acts trivially on $\mathcal Q_0$, \eqref{special action} implies that $s$ acts trivially on $\mathcal M_0$.  This however contradicts the fact that the action $G \curvearrowright \mathcal M_0$ is faithful. Thus, we have $ \C1_{\overline V_\theta} \ovt \mathcal Z(\mathcal Q_0) \neq \mathcal Z(\iota_0(\mathcal M_0))$.
\end{proof}

\begin{claim}\label{claim:final}
The action $G \curvearrowright \mathcal Z(\mathcal M_0)$ is not $\varphi$-preserving.
\end{claim}
\begin{proof}[Proof of Claim \ref{claim:final}]
By contradiction, assume that the action $G \curvearrowright \mathcal Z(\mathcal M_0)$ is $\varphi$-preserving. Regard $\iota_0(\mathcal Z(\mathcal M_0)) \subset \iota_0(\mathcal M_0) \subset \rL^\infty(\oV) \ovt \cN_0$. The same argument as in the proof of Claim \ref{claim:nonpreserving} shows that for every $x \in \mathcal Z(\mathcal M_0)$, the bounded measurable function $\overline V \to \C : \overline v \mapsto \psi(\iota_0(x)(\overline v))$ is $\nu_{\overline V}$-almost everywhere constant. Since $(\nu_{\overline V} \otimes \psi) \circ \iota_0 = \varphi$, this means that for every $x \in \mathcal Z(\mathcal M_0)$, we have $(\id_{\overline V} \otimes \psi)(\iota_0(x)) = \varphi(x) \,  1_{\overline V}$.

Recall that $\rL^\infty(\overline V) = \rL^\infty(\overline V_\theta) \ovt  \rL^\infty(\overline U_\theta)$, $1_{\overline V} = 1_{\overline V_\theta} \otimes 1_{\overline U_\theta}$, $\id_{\overline V} = \id_{\overline V_\theta} \otimes \id_{\overline U_\theta}$. Once again, consider the splitting $\rL^\infty(\oV) \ovt \cN_0 = \rL^\infty(\overline V_\theta) \ovt (\rL^\infty(\overline U_\theta) \ovt \mathcal N_0)$. Choose a Borel probability measure $\nu_{\overline U_\theta} \in \Prob(\overline U_\theta)$ that is in the same class as the Haar measure $m_{\overline U_\theta}$. Set $\psi_0 = \nu_{\overline U_\theta} \otimes \psi \in (\rL^\infty(\overline U_\theta) \ovt \mathcal N_0)_\ast$. Combining the result of the previous paragraph with \eqref{eq:tensor-slice}, we obtain
\begin{align*}
\forall x \in \mathcal Z(\mathcal M_0), \quad (\id_{\overline V_\theta} \otimes \psi_0)(\iota_0(x)) & = (\id_{\overline V_\theta} \otimes \nu_{\overline U_\theta} \otimes \psi)(\iota_0(x)) \\
& = (\id_{\overline V_\theta} \otimes \nu_{\overline U_\theta})\left((\id_{\overline V_\theta} \otimes \id_{\overline U_\theta} \otimes \psi)(\iota_0(x)) \right) \\
& = (\id_{\overline V_\theta} \otimes \nu_{\overline U_\theta} )\left(\varphi(x) \, 1_{\overline V}\right) \\
&= \varphi(x) \, 1_{\overline V_\theta}.
\end{align*}

Observe that $ \iota_0(\mathcal Z(\mathcal M_0)) = \mathcal Z(\iota_0(\mathcal M_0))$. Recall that we have $\mathcal Z(\mathcal Q_0) \subset \rL^\infty(\overline U_\theta) \ovt \mathcal N_0$ and $\C1_{\overline V_\theta} \ovt \mathcal Z(\mathcal Q_0) \subset \iota_0(\mathcal Z(\mathcal M_0)) \subset \rL^\infty(\overline V_\theta) \ovt \mathcal Z(\mathcal Q_0)$ by Claim \ref{claim:inclusion}. Denote by $z_0 \in  \mathcal Z(\mathcal Q_0)$ the support projection of ${\psi_0}|_{ \mathcal Z(\mathcal Q_0)}$. Since $\C1_{\overline V_\theta} \ovt \mathcal Z(\mathcal Q_0) \subset \iota_0(\mathcal Z(\mathcal M_0))$, let $p_0 \in \mathcal Z(\mathcal M_0)$ be the unique projection such that $1_{\overline V_\theta} \otimes z_0 = \iota_0(p_0)$. Then the above computation applied to $x = p_0$ shows that
$$1_{\overline V_\theta} = (\id_{\overline V_\theta} \otimes \psi_0)(1_{\overline V_\theta} \otimes z_0) = (\id_{\overline V_\theta} \otimes \psi_0)(\iota_0(p_0)) = \varphi(p_0) \, 1_{\overline V_\theta}$$ 
and so $\varphi(p_0) = 1$. Since $\varphi$ is faithful, we have $p_0 = 1$ and so $z_0 = 1$. Thus, $\psi_0|_{\mathcal Z(\mathcal Q_0)}$ is faithful. 

As we mentioned before, for every $x \in \mathcal Z(\mathcal M_0)$, we have $(\id_{\overline V_\theta} \otimes \psi_0)(\iota_0(x)) \in \C1_{\overline V_\theta}$. Since $\C1_{\overline V_\theta} \ovt \mathcal Z(\mathcal Q_0) \subset \iota_0(\mathcal Z(\mathcal M_0))$, for every $x \in \mathcal Z(\mathcal M_0)$ and every $b \in \mathcal Z(\mathcal Q_0)$, we have $\iota_0(x)(1_{\overline V_\theta} \otimes b)  \in \iota_0(\mathcal Z(\mathcal M_0))$ and so $(\id_{\overline V_\theta} \otimes b\psi_0)(\iota_0(x)) = (\id_{\overline V_\theta} \otimes \psi_0)(\iota_0(x) (1_{\overline V_\theta} \otimes b) ) \in \C 1_{\overline V_\theta}$. Since $ \psi_0|_{\mathcal Z(\mathcal Q_0)}$ is faithful, Hahn--Banach theorem implies that the linear subspace $\left\{b \psi_0 \, \mid \, b \in \mathcal Z(\mathcal Q_0)\right \}$ is $\|\, \cdot \, \|$-dense in $\mathcal Z(\mathcal Q_0)_\ast$. This implies that for every $x \in \mathcal Z(\mathcal M_0)$ and every $\rho \in \mathcal Z(\mathcal Q_0)_\ast$, we have $(\id_{\overline V_\theta} \otimes \rho)(\iota_0(x)) \in \C 1_{\overline V_\theta}$. Moreover, for every $x \in \mathcal Z(\mathcal M_0)$ and every $\rho \in \rL^\infty(\overline V_\theta)_\ast$, we have $(\rho \otimes \id_{\mathcal Z(\mathcal Q_0)})(\iota_0(x)) \in \mathcal Z(\mathcal Q_0)$. Then \cite[Theorem B]{GK95} implies that $\iota_0(\mathcal Z(\mathcal M_0)) = \C 1_{\overline V_\theta} \ovt \mathcal Z(\mathcal Q_0)$ and thus $\mathcal Z(\iota_0(\mathcal M_0)) = \C 1_{\overline V_\theta} \ovt \mathcal Z(\mathcal Q_0)$. This contradicts Claim \ref{claim:factor}.
\end{proof}

By Claim \ref{claim:final}, we have that $(\mathcal Z(\mathcal M_0), \varphi)$ is an abelian ergodic $(G, \mu)$-von Neumann algebra for which the action $G \curvearrowright \mathcal Z(\mathcal M_0)$ is not $\varphi$-preserving. We can now apply \cite[Theorem 1]{NZ00} to obtain that there exist a proper parabolic subgroup $Q \subset G$ and a $G$-equivariant normal unital $\ast$-embedding $\Theta : \rL^\infty(G/Q) \to \mathcal Z(\mathcal M_0) \subset \mathcal M$ such that $\varphi \circ \Theta = \nu_Q$. This finishes the proof of Theorem \ref{thm:NZ}.
\end{proof}

\begin{rem}\label{simple case}
In order to prove Theorem \ref{thm:NZ} in the case when $G$ has trivial center, it is actually sufficient to prove it when $G$ is simple. Indeed, assume that the result is true for simple groups and assume that $(\cM,\varphi)$ is an ergodic $(G,\mu)$-von Neumann algebra, where $\mu$ is an admissible left $K$-invariant measure on $G$ and $\varphi$ is faithful. Without loss of generality, we may assume that $\mu$ splits as a product of measures on each of the simple factors of $G$. Assume that $\varphi$ is not $G$-invariant. Then it is not invariant under at least one simple factor $G_1$ of $G$. Write $G = G_1 \times G_2$. By \cite[Lemma 2.8]{BBHP20}, there exists a $G_1$-equivariant, $\varphi$-preserving conditional expectation $\cM \to \cM^{G_2}$. Thus, $\varphi$ is not $G_1$-invariant on $\cM^{G_2}$ and we may apply Theorem \ref{thm:NZ} to the ergodic action $(G_1,\mu_1) \actson (\cM^{G_2},\varphi)$, to find a proper parabolic subgroup $Q_1 < G_1$ and a $G_1$-equivariant von Neumann morphism $\mathrm L^\infty(G_1/Q_1) \to \cM^{G_2} \subset \cM$. The conclusion follows once we realize that $G_1/Q_1 \simeq G/Q$ as measurable $G$-spaces, where $Q = Q_1 \times G_2$ is a proper parabolic subgroup of $G$.
\end{rem}

\section{Proofs of the main results}\label{section:results}

In this section, we use the results obtained in Sections \ref{section:induced} and \ref{section:NZ} to prove the main results stated in the introduction. We assume that $G$ is a connected semisimple Lie group with finite center and no nontrivial compact factor, all of whose simple factors have real rank at least two. We use the notation from the introduction. Denote by $Z(G)$ the center of $G$.

\subsection{Proof of Theorem \ref{main NCNZ}}

Applying our Theorem \ref{induced stationary state} on the construction of the induced stationary state, we derive Theorem \ref{main NCNZ} from Theorem \ref{thm:NZ} and a disintegration argument as follows.

\begin{proof}[Proof of Theorem \ref{main NCNZ}]
Let $\Gamma < G$ be any lattice. Assume that $\phi$ is not $\Gamma$-invariant. Then we may find a $\Gamma$-invariant von Neumann subalgebra of $M$ with separable predual on which $\phi$ is not $\Gamma$-invariant. So we may as well assume that $M$ has separable predual.
Consider the induced von Neumann algebra $\cM = \Ind_\Gamma^G(M)=\rL^\infty(G/\Gamma) \ovt M$ with the corresponding induced $G$-action arising from a measurable section $\tau: G/\Gamma \to G$ and the corresponding measurable $1$-cocycle $c: G \times G/\Gamma \to \Gamma$. We assume that $\tau(\Gamma) = e$. As we mentioned in Subsection \ref{subsection:vN}, the induced $G$-action is ergodic since the initial $\Gamma$-action is ergodic. By Theorem \ref{induced stationary state}, there exists a $\mu$-stationary faithful normal state on $\cM$, which is not $G$-invariant. So Theorem \ref{thm:NZ} implies that there exist a proper parabolic subgroup $P \subset Q \subsetneq G$ and a $G$-equivariant unital $\ast$-homomorphism $\Theta: \rC(G/Q) \to \cM$. We simply denote  by $\sigma$ the actions $G \curvearrowright \rC(G/Q)$ and $\Gamma \curvearrowright M$.

Since $A = \rC(G/Q)$ is a separable unital $\rC^*$-algebra, \cite[Theorem IV.8.25]{Ta02} implies that there exists an essentially unique measurable field of unital $\ast$-homomorphisms $\{\pi_x : A \to M\}_{x \in G/\Gamma}$  such that $\Theta(a) = \int_{G/\Gamma}^\oplus \pi_x(a) \, {\rm d}m_{G/\Gamma}(x)$, for all $a \in A$. 
\begin{claim}\label{claim:measurability}
For every $a \in A$, the map $G \times G/\Gamma \to M : (g, x) \mapsto \sigma_{c(g,x)^{-1}}(\pi_{gx}(\sigma_g(a)))$ is measurable.
\end{claim}

\begin{proof}[Proof of Claim \ref{claim:measurability}]
The claim will follow after seeing the desired map as the composition of several measurable maps.
\begin{itemize}
\item Obviously, the maps $(g,x) \mapsto c(g,x)^{-1} \in \Gamma$ and $(g,x) \mapsto gx \in G/\Gamma$ are measurable on $G \times G/\Gamma$;
\item The action map $\Gamma \times M \to M$ is weakly continuous, hence it is measurable as well;
\item It remains to check that the map $G \times G/\Gamma \to M : (g,x) \mapsto \pi_x(\sigma_g(a))$ is measurable. Since the orbit map $G \mapsto A: g \mapsto \sigma_g(a)$ is continuous, this boils down to proving that $G/\Gamma \times A \to M: (x,a) \mapsto \pi_x(a)$ is measurable.
\end{itemize}
Denote by $\Rep(A, H)$ the set of all $\ast$-representations of $A$ on $H$ endowed with the pointwise topology of $\ast$-strong convergence in $\mathbf B(H)$. Since $A$ is separable, since $\rC^*$-representations have norm at most one, and since closed bounded balls in $\mathbf B(H)$ are $\ast$-strongly Polish, $\Rep(A,H)$ is a Polish space. By assumption, the map $G/\Gamma \to \Rep(A, H) : x \mapsto \pi_x$ is measurable. We claim that the map $A \times \Rep(A,H) \to \mathbf B(H): (a,\pi) \mapsto \pi(a)$ is continuous.  Take a net $(a_i, \pi_i) \in A \times \Rep(A,H)$, converging to an element $(a,\pi)$. For every $\xi \in H$, we have
\[\lim_i \|\pi(a)\xi - \pi_i(a_i) \xi\| \leq \lim_i \|(\pi - \pi_i)(a)\xi\| + \lim_i \|\pi_i(a - a_i)\xi\| \leq 0 + \lim_i \|a - a_i\| \cdot\|\xi\| = 0.\]
So indeed the map $(a,\pi) \mapsto \pi(a)$ is continuous and it follows by composition that the map $(x,a) \mapsto \pi_x(a) \in M$ is measurable on $G/\Gamma \times A$. This is what we wanted.
\end{proof}

%

Since $\Theta$ is $G$-equivariant and since the field $\{\pi_x : A \to M\}_{x \in G/\Gamma}$ is essentially unique, for every $g \in G$, every $a \in \rC(G/Q)$ and almost every $x \in G/\Gamma$, we have 
\begin{equation}\label{ae equiv}
\pi_x(a) = \sigma_{c(g,x)^{-1}}(\pi_{gx}(\sigma_g(a))).
\end{equation}
Since $A$ is separable, a continuity argument allows to change the quantifiers and conclude that for every $g \in G$ and almost every $x \in G/\Gamma$, Equation \eqref{ae equiv} holds simultaneously for every $a \in A$. Now, we conclude by a variation of the proof of \cite[Proposition B.5]{Zi84} that it is possible to modify the maps $\pi_x$ on a null set of $x$'s to actually have equation \eqref{ae equiv} hold everywhere. We give all the details for completeness.

Denote by $\Ball_A(0, 1)$ (resp.\ $\Ball_M(0, 1)$) the closed ball in $A$ (resp.\ $M$) of center $0$ and radius $1$ with respect to the uniform norm. Since $M$ has separable predual, $\Ball_M(0, 1)$ is a Polish space with respect to the $\ast$-strong topology and so it can be realized as a Borel subset of $[0,1]$. Using Claim \ref{claim:measurability} and Fubini's theorem, we conclude that the following subset 
\[X_0 = \left\{x \in G/\Gamma \mid  \forall a \in \Ball_A(0, 1), \text{ the map } g \mapsto \sigma_{c(g,x)^{-1}}(\pi_{gx}(\sigma_g(a))) \text{ is ess.\ constant on } G \right\}\]
is measurable and conull in $G/\Gamma$. For every $x \in X_0$ and every $a \in \Ball_A(0, 1)$, denote by $\rho_x(a) \in M$ the essential value of the measurable map $G \to \Ball_M(0, 1) : g \mapsto \sigma_{c(g,x)^{-1}}(\pi_{gx}(\sigma_g(a)))$. Choose a Borel probability measure $\eta \in \Prob(G)$ in the same class as the Haar measure. Recall that we view $\Ball_M(0, 1) \subset [0, 1]$ as a Borel subset. Then for every $a \in \Ball_A(0, 1)$, Claim \ref{claim:measurability} and Fubini's theorem imply that the map $X_0 \to \Ball_M(0, 1) : x  \mapsto \rho_x(a)$ where $\rho_x(a) = \int_G  \sigma_{c(g,x)^{-1}}(\pi_{gx}(\sigma_g(a))) \, {\rm d}\eta(g)$ is measurable and coincides $m_{G/\Gamma}$-almost everywhere with the measurable map $X_0 \to \Ball_M(0, 1) : x \mapsto \pi_x(a)$. 

For all $x \in G/\Gamma$, all $g,h \in G$ and all $a \in \Ball_A(0, 1)$, the $1$-cocycle relation for $c$ gives the formula 
\begin{equation}\label{eq:cocycle}
\sigma_{c(h,x)^{-1}c(g,hx)^{-1}}(\pi_{ghx}(\sigma_{gh}(a))) = \sigma_{c(gh,x)^{-1}}(\pi_{ghx}(\sigma_{gh}(a))).
\end{equation}
In particular, if $x \in X_0$ and $h \in G$ are fixed, then the right hand side of \eqref{eq:cocycle} is essentially constant in the variable $g \in G$, for all $a \in \Ball_A(0, 1)$. This implies that $hx \in X_0$ and 
\begin{equation}\label{ae equiv 2}\sigma_{c(h,x)^{-1}}(\rho_{hx}(\sigma_h(a))) = \rho_x(a).\end{equation}
Since $X_0 \subset G/\Gamma$ is conull and $G$-invariant, we necessarily have $X_0 = G/\Gamma$. Note that Equation \eqref{ae equiv 2} holds for every $x \in G/\Gamma$, every $h \in G$ and every $a \in A$.

Take $x = \Gamma \in G/\Gamma$ and consider the unital $\ast$-homomorphism $\theta: A \to M : a \mapsto \rho_x(a)$. We claim that $\theta$ is $\Gamma$-equivariant. 
Fix $a \in A$ and $\gamma \in \Gamma$. Note that $\gamma x = x$ and $c(\gamma,x) =  \gamma$.  Applying Equation \eqref{ae equiv 2}, we get
\[\theta(\sigma_\gamma (a)) = \rho_x(\sigma_{\gamma}(a))= \rho_{\gamma x}(\sigma_\gamma(a)) = \sigma_{c(\gamma,x)}(\rho_x(a)) = \sigma_\gamma (\theta(a)).\]
Thus, $\theta$ is $\Gamma$-equivariant. Since the state $\phi \circ \theta$ is $\mu_0$-stationary on $\rC(G/Q)$ and since $\nu_Q$ is the only $\mu_0$-stationary state on $\rC(G/Q)$ (see \cite{Fu73, GM89}), we have $\phi \circ \theta = \nu_Q$. Then $\theta$ extends to a $\Gamma$-equivariant normal unital $\ast$-embedding $\rL^\infty(G/Q) \to M$ such that $\phi \circ \theta = \nu_Q$.
%
\end{proof}

\subsection{Proof of Theorem \ref{main stationary characters}}

Before proving Theorem \ref{main stationary characters}, we prove a few preliminary results. The following useful result is essentially contained in \cite[Remark 13]{Oz16} (see also \cite[\S 7]{CP12}). 

\begin{lem}\label{lem:free}
Let $\Gamma < G$ be any irreducible lattice and $H < G$ any proper closed subgroup.  Set $Z(\Gamma) = \Gamma \cap Z(G)$. Denote by $\nu_{G/H} \in \Prob(G/H)$ a $G$-quasi-invariant Borel probability measure. Then for every $\gamma \in \Gamma \setminus Z(\Gamma)$, $\nu_{G/H} (\{y \in G/H \mid \gamma y = y\}) = 0$.
\end{lem}

\begin{proof}
We will use the classical fact that the intersection of the irreducible lattice $\Gamma$ with any proper normal closed subgroup of $G$ is contained in the center of $G$. To see this general fact, we can mod out by the center of $G$ to reduce to the center-free case. Then if $N$ is a proper normal closed subgroup of $G$, we have in fact a splitting $G \simeq N \times N'$. Now $\Gamma \cap N$ is normalized by $\Gamma$, while it commutes with $N'$. So it is also normalized by the projection of $\Gamma$ inside $N \simeq G/N'$, which is dense in $N$. Since $\Gamma \cap N$ is closed in $N$, it must in fact be normalized by the whole of $N$, by continuity. Since $\Gamma$ is discrete, $\Gamma \cap N$ must be central in $N$, as claimed.

We will apply this general fact to $N = \bigcap_{g \in G} gHg^{-1}$, which is indeed a proper closed normal subgroup of $G$. 
Let $\gamma \in \Gamma$ be any element such that $\nu_{G/H} (\{gH \in G/H \mid \gamma gH = gH\}) > 0$. Let $d = \dim(G)$. By \cite[Remark 13]{Oz16}, for almost every $(g_1, \dots, g_{d + 1}) \in G^{d + 1}$, we have $\bigcap_{i = 1}^{d + 1} g_i H g_i^{-1} = N$. This implies that $\gamma \in \Gamma \cap N$ and so $\gamma \in Z(\Gamma)$. 
\end{proof}

Observe that when $G$ has trivial center, Lemma \ref{lem:free} implies that the action $\Gamma \curvearrowright (G/H, \nu_{G/H})$ is essentially free.

The following useful lemma is inspired by observations due to Hartman--Kalantar (see \cite[Example 4.11]{HK17}) and Haagerup \cite[Lemma 3.1]{Ha15}.

\begin{lem}\label{HK dirac}
Let $\Gamma < G$ be any irreducible lattice and $\mu_0 \in \Prob(\Gamma)$ any Furstenberg probability measure as in the introduction.  Set $Z(\Gamma) = \Gamma \cap Z(G)$. Let $B$ be any unital $\rC^*$-algebra and $\pi: \Gamma \to \cU(B)$ any unitary representation. Consider the conjugation action $\sigma : \Gamma \actson B$ defined by $\sigma_\gamma = \Ad(\pi(\gamma))$ for every $\gamma \in \Gamma$.

Assume that there exists a proper parabolic subgroup $P \subset Q \subsetneq G$ and a $\Gamma$-equivariant unital $\ast$-homomorphism $\theta: \rC(G/Q) \to B$. Let $\phi \in \mathcal S(B)$ be any $\mu_0$-stationary state. Then for every $\gamma \in \Gamma \setminus Z(\Gamma)$, we have $\phi(\pi(\gamma)) = 0$. 
\end{lem}

\begin{proof}
Denote by $A \subset B$ the {\em separable} unital $\rC^*$-subalgebra generated by $\pi(\Gamma)$ and $\theta(\rC(G/Q))$ and observe that $A \subset B$ is globally $\Gamma$-invariant under the action $\sigma$. Since $A$ is separable and since $\phi|_A \in \mathcal S(A)$ is $\mu_0$-stationary, we may apply Theorem \ref{thm:poisson-map} to obtain a $\Gamma$-equivariant measurable map $\beta_\phi : G/P \to \mathcal S(A): w \mapsto \phi_{w} $ which satisfies $\phi = \int_{G/P} \phi_w \, {\rm d}\nu_P(w)$. Denote by $p_Q: G/P \to G/Q$ the projection map that moreover satisfies $(p_Q)_\ast \nu_P = \nu_Q$. Since the state $\phi \circ \theta$ is $\mu_0$-stationary on $\rC(G/Q)$ and since $\nu_Q$ is the only $\mu_0$-stationary state on $\rC(G/Q)$ (see \cite{Fu73, GM89}), we have $\phi \circ \theta = \nu_Q$. Since $(G/Q,\nu_Q)$ is a $(\Gamma, \mu_0)$-boundary in the sense of Furstenberg, we even deduce that there exists a conull measurable subset $\Omega_1 \subset G/P$ such that for every $w \in \Omega_1$, the state $\phi_w \circ \theta \in \mathcal S(\rC(G/Q))$ is equal to the Dirac state $\delta_{p_Q(w)}$ (see \cite[Theorem 2.14]{BS04}). Since $\delta_{p_Q(w)} \in \mathcal S(\rC(G/Q))$ is multiplicative, we infer that for all $w \in \Omega_1$, $\theta(\rC(G/Q))$ lies in the multiplicative domain of $\phi_w$ (see e.g.\ \cite[Proposition 1.5.7]{BO08}).

Fix $\gamma \in \Gamma \setminus Z(\Gamma)$. Since $\nu_Q(\{y \in G/Q \mid \gamma y = y\}) = 0$ by Lemma \ref{lem:free}, we may find a conull measurable subset $\Omega_2 \subset G/P$ such that $\gamma p_Q( w) \neq p_Q(w)$ for all $w \in \Omega_2$. Fix $w \in \Omega_1 \cap \Omega_2$, set $y = p_Q(w) \in G/Q$ and choose a continuous function $f \in \rC(G/Q)$ such that $f(y) = 1$ and $f(\gamma y) = 0$. We compute
\[\phi_w(\pi(\gamma)) = f(y) \, \phi_w(\pi(\gamma)) = \phi_w(\theta(f)\pi(\gamma)) = \phi_w(\pi(\gamma) \theta(\sigma_\gamma^{-1}(f))) = \phi_w(\pi(\gamma)) \, f(\gamma y) = 0.\]
By integrating with respect to $w \in G/P$, we obtain $\phi(\pi(\gamma)) = 0$. 
\end{proof}


We are now ready to prove Theorem \ref{main stationary characters}.

\begin{proof}[Proof of Theorem \ref{main stationary characters}]
Let $\Gamma < G$ be any irreducible lattice. Observe that the set $\mathcal C(\Gamma, \mu_0)$ of all $\mu_0$-characters on $\Gamma$ is a nonempty compact convex subset of the space $\ell^\infty(\Gamma)$ endowed with the weak$^*$-topology. Using Krein--Milman's theorem, in order to prove Theorem \ref{main stationary characters}, it suffices to show that any extreme point in $\mathcal C(\Gamma, \mu_0)$ is conjugation invariant.

Let $\varphi \in \mathcal C(\Gamma, \mu_0)$ be any extreme point. Denote by $(\pi_0,H_0,\xi_0)$ the GNS triple corresponding to $\varphi$ and set $M = \pi_0(\Gamma)\dpr$. Recall that $\varphi(\gamma) = \langle \pi_0(\gamma)\xi_0,\xi_0\rangle$ for every $\gamma \in \Gamma$. We denote by $\phi$ the normal state $\langle \, \cdot \, \xi_0, \xi_0\rangle$ on $M$ and observe that $\varphi = \phi \circ \pi_0$. Denote by $\sigma : \Gamma \curvearrowright M$ the conjugation action defined by $\sigma_\gamma = \Ad(\pi_0(\gamma))$ for every $\gamma \in G$. Then $\phi \in M_\ast$ is a normal $\mu_0$-stationary state. Observe that $\varphi \in \mathcal C(\Gamma, \mu_0)$ is conjugation invariant if and only if $\phi \in M_\ast$ is $\Gamma$-invariant.

First, we prove that $\phi \in M_\ast$ is faithful. Indeed, let $x \in M$ be any element such that $\phi(x^*x) = 0$. Since $\mu_0 \ast \phi = \phi$, we obtain
\[\sum_{\gamma \in \Gamma} \mu_0(\gamma)\,\Vert x\pi_0(\gamma)\xi_0\Vert^2 = \sum_{\gamma \in \Gamma} \mu_0(\gamma)\,\phi(\pi_0(\gamma)^*x^*x\pi_0(\gamma)) = (\mu_0 \ast \phi)(x^*x) = \phi(x^*x) = 0.\]
This implies that $ x\pi_0(\gamma)\xi_0 = 0$ for all $\gamma \in \supp(\mu_0) = \Gamma$. Since $\xi_0$ is $\pi_0(\Gamma)$-cyclic, we conclude that $x \xi = 0$ for all $\xi \in H_0$ and so $x = 0$. Thus, $\phi \in M_\ast$ is a faithful normal state. 


Next, we prove that the action $\Gamma \curvearrowright M$ is ergodic. We prove the contrapositive statement. Assume that $M^\Gamma = \cZ(M)$ is nontrivial. Then it contains a nontrivial projection $p \in \cZ(M)$. Since $\phi$ is faithful,  $\phi(p) \notin \{0,1\}$. Define the $\mu_0$-stationary normal states $\phi_1: x \mapsto  \frac{1}{\phi(p)}\phi(xp)$ and $\phi_2: x \mapsto \frac{1}{\phi(1 - p)}\phi(x (1 - p))$, for $x \in M$. We have $\phi = \phi(p) \, \phi_1 + \phi(1-p) \, \phi_2$. We have $\phi_1(p) = 1$ and $\phi_2(p) = 0$ so that $\phi_1 \neq \phi_2$. Define the $\mu_0$-characters $\varphi_1 = \phi_1\circ\pi_0 \in \mathcal C(\Gamma, \mu_0)$ and $\varphi_2 = \phi_2\circ\pi_0 \in \mathcal C(\Gamma, \mu_0)$. We have $\varphi = \phi(p) \, \varphi_1 + \phi(1 - p) \, \varphi_2$. Since $\phi_1$ and $\phi_2$ are normal, since $\phi_1 \neq \phi_2$ and since the linear span of $\pi_0(\Gamma)$ is ultraweakly dense in $M$, there exists $\gamma \in \Gamma$ such that $\varphi_1 (\gamma) = \phi_1(\pi_0(\gamma)) \neq \phi_2(\pi_0(\gamma)) = \varphi_2(\gamma)$. This implies that $\varphi_1 \neq \varphi_2$ and hence $\varphi$ is not an extreme point in $\mathcal C(\Gamma, \mu_0)$. This shows that the action $\Gamma \curvearrowright M$ is ergodic.

Then the action $\Gamma \actson M$ is ergodic and $\phi \in M_\ast$ is a $\mu_0$-stationary faithful normal state. Assume by contradiction that $\phi$ is not $\Gamma$-invariant. By Theorem \ref{main NCNZ}, there exist a proper parabolic subgroup $P \subset Q \subsetneq G$ and a $\Gamma$-equivariant unital $\ast$-homomorphism $\theta: \rC(G/Q) \to M$. By Lemma \ref{HK dirac}, we obtain that $\varphi = \phi \circ \pi_0$ is supported on the center of $\Gamma$, hence is conjugation invariant. This further implies that $\phi$ is $\Gamma$-invariant, contradicting our assumption.
\end{proof}

\subsection{Proof of Theorem \ref{main peterson}}
Assume that $G$ has trivial center. Let $\Gamma < G$ be any irreducible lattice. Let $\varphi$ be any extreme point in the space of characters of $\Gamma$. Denote by $(\pi_0,H_0,\xi_0)$ the GNS triple corresponding to $\phi$. As explained in the proof of Theorem \ref{main stationary characters}, the von Neumann algebra $M = \pi_0(\Gamma)\dpr$ is a finite factor and the vector state $\phi = \langle \, \cdot \, \xi_0, \xi_0\rangle \in M_\ast$ is the canonical faithful normal trace. Denote by $J: H_0 \to H_0 : x\xi_0 \mapsto x^*\xi_0$ the canonical anti-unitary. We have $JMJ = M' \cap \mathbf B(H_0)$. The Hilbert space $H_0$ is naturally endowed with an $M$-$M$-bimodule structure given by $x \eta y = x Jy^*J \eta$ for all $x, y \in M$ and all $\eta \in H_0$. Note that we have $x \xi_0 = J x^* J \xi_0$ for every $x \in M$.

The key aspect of Peterson's approach \cite{Pe14} is to study the {\em noncommutative} Poisson boundary, defined as the fixed-point von Neumann algebra $\cB = \left( \rL^\infty(G/P) \ovt \mathbf B(H_0) \right)^\Gamma$ with respect to the $\Gamma$-action $\alpha: \Gamma \actson \rL^\infty(G/P,\nu_P) \ovt \mathbf B(H_0)$ defined by $\alpha_\gamma =\sigma_\gamma \otimes \Ad(J\pi_0(\gamma)J)$ for every $\gamma \in \Gamma$. Here, $\sigma: \Gamma \actson \rL^\infty(G/P,\nu_P)$ is the natural translation action. Alternatively, we can view $\cB$ as the von Neumann algebra of all essentially bounded measurable functions $f : G/P \to \mathbf B(H_0)$, modulo equality $\nu_P$-almost everywhere, that satisfy $f(\gamma w) = \Ad(J\pi_0(\gamma)J)(f(w))$ for every $\gamma \in \Gamma$ and $\nu_P$-almost every $w \in G/P$. Note that $\C 1 \ovt M \subset \mathcal B$ corresponds to the von Neumann subalgebra of all essentially constant measurable functions $f : G/P \to M$. Recall that $\mathcal B$ is amenable (see \cite[Section 2]{CP13}).

Define the conjugation action $\beta : \Gamma \curvearrowright  \rL^\infty(G/P) \ovt \mathbf B(H_0)$ by $\beta_\gamma = \Ad(1 \otimes \pi_0(\gamma))$ for every $\gamma \in \Gamma$. Then $\beta$ commutes with $\alpha$ and so $\cB$ is globally $\Gamma$-invariant under the action $\beta$. The next lemma will allow us to apply Theorem \ref{main NCNZ}.

\begin{lem}\label{lem:beta}
Keep the same notation as above. The following assertions hold true.
\begin{itemize}
\item [$(\rm i)$] The action $\beta : \Gamma \actson \cB$ is ergodic.
\item [$(\rm ii)$] The normal state $\Phi : \cB \to \C : f \mapsto \int_{G/P} \langle f(w)\xi_0,\xi_0\rangle \, {\rm d}\nu_P(w)$ is $\mu_0$-stationary.
\end{itemize}
\end{lem}
\begin{proof}
$(\rm i)$ Denote by $\cB^{\beta(\Gamma)} = \{f \in \mathcal B \mid \beta_\gamma(f) = f, \forall \gamma \in \Gamma\}$ the fixed-point von Neumann subalgebra of $\mathcal B$ with respect to the action $\beta$. Since $JMJ = M' \cap \mathbf B(H_0)$, by construction,  $\cB^{\beta(\Gamma)}$ is the von Neumann algebra of all  essentially bounded measurable functions $f : G/P \to JMJ$, modulo equality $\nu_P$-almost everywhere, that satisfy $f(\gamma w) = \Ad(J\pi_0(\gamma)J)(f(w))$ for every $\gamma \in \Gamma$ and $\nu_P$-almost every $w \in G/P$. Since $JMJ$ is a finite von Neumann algebra with separable predual, we may view it as a separable metric space with respect to the distance $d : JMJ \times JMJ \to \R_{\geq 0}$ defined by $d(JxJ, JyJ) = \phi((y - x)^*(y - x))^{1/2}$ for all $x, y \in M$. Moreover, the  action $\Ad (J \pi_0(\, \cdot \, ) J) : \Gamma \curvearrowright (JMJ, d)$ is isometric since the map $JMJ \to \C : JxJ \mapsto \phi(x^* )$ is a (faithful normal) trace on $JMJ$. Then \cite[Theorem 2.5]{BF14} and the fact that $M$ is a factor imply that $\mathcal B^{\beta(\Gamma)} \subset \C 1$. Thus, the action $\beta : \Gamma \actson \cB$ is ergodic.

$(\rm ii)$ For every $f \in \cB$, we have
\begin{align*}
 \sum_{\gamma \in \Gamma} \mu_0(\gamma) \Phi(\beta_\gamma^{-1}(f)) & = \sum_{\gamma \in \Gamma} \mu_0(\gamma) \int_{G/P}\langle f( w)\pi_0(\gamma)\xi_0,\pi_0(\gamma)\xi_0\rangle \, {\rm d}\nu_P(w) \\
& = \sum_{\gamma \in \Gamma} \mu_0(\gamma) \int_{G/P} \langle f(w)J\pi_0(\gamma)^*J\xi_0,J\pi_0(\gamma)^*J\xi_0\rangle \, {\rm d}\nu_P(w) \\
& =  \sum_{\gamma \in \Gamma} \mu_0(\gamma) \int_{G/P} \langle f(\gamma w)\xi_0,\xi_0\rangle \, {\rm d}\nu_P(w) \quad (\text{since } f \in \mathcal B)\\ 
&= \int_{G/P} \langle f( w)\xi_0,\xi_0\rangle \, {\rm d}\nu_P(w) \quad (\text{since } \nu_P \text{ is } \mu_0\text{-stationary})\\ 
& = \Phi(f).
\end{align*}
Thus, $\Phi$ is a $\mu_0$-stationary state. 
\end{proof}

\begin{proof}[Proof of Theorem \ref{main peterson}]
Let $\varphi$ be any extreme point in the space of characters of $\Gamma$. Keep the same notation as above. By Lemma \ref{lem:beta} and Theorem \ref{main NCNZ}, the following dichotomy holds:
\begin{enumerate}
\item Either $\Phi$ is $\Gamma$-invariant with respect to $\beta$. 
\item Or there exist a proper parabolic subgroup $P \subset Q \subsetneq G$ and a $\Gamma$-equivariant map $\Theta : \rC(G/Q) \to \cB$.
\end{enumerate}

If $(2)$ holds, Lemma \ref{HK dirac} shows that $\varphi = \phi \circ \pi_0$ must be the Dirac map at the identity. If $(1)$ holds, we show that that $M$ is a finite dimensional factor and hence $\varphi$ is almost periodic. Since $\Gamma$ has property (T), $M$ has property (T) in the sense of \cite{CJ83} so it suffices to prove that $M$ is amenable. As we saw, $\cB$ is amenable so we only need to verify that $\C1 \ovt M = \cB$.

For every $f \in \cB$ and every $\gamma \in \Gamma$, we have
\begin{align*}
\Phi(\beta_\gamma^{-1}(f)) & = \int_{G/P}\langle f( w)\pi_0(\gamma)\xi_0,\pi_0(\gamma)\xi_0\rangle \, {\rm d}\nu_P(w) \\
&=  \int_{G/P} \langle f(w)J\pi_0(\gamma)^*J\xi_0,J\pi_0(\gamma)^*J\xi_0\rangle \, {\rm d}\nu_P(w) \\
&=\int_{G/P} \langle f(\gamma w)\xi_0,\xi_0\rangle \, {\rm d}\nu_P(w).
\end{align*}
Since this quantity does not depend on $\gamma \in \Gamma$, the bounded $\mu_0$-harmonic function $\Gamma \to \C : \gamma \mapsto \int_{G/P} \langle f(\gamma w)\xi_0,\xi_0\rangle \, {\rm d}\nu_P(w)$ is constant. Since $(G/P, \nu_P)$ is the $(\Gamma, \mu_0)$-Poisson boundary,  Theorem \ref{thm:harmonicity} implies that the function $G/P \to \C : w \mapsto\langle f( w)\xi_0,\xi_0\rangle$ is $\nu_P$-almost everywhere constant. Since $\C1 \ovt M \subset \cB$, we deduce that for all $f \in \cB$ and all $a,b \in M$,  $(1 \otimes b^*) f (1 \otimes a) \in \mathcal B$ and so the measurable function $G/P \to \C : w \mapsto \langle f(w)a\xi_0,b\xi_0\rangle$ is essentially constant. By separability of $H_0$ and density of $M\xi_0$ in $H_0$, we conclude that $f$ is essentially constant. Since $f(\gamma w) = \Ad(J\pi_0(\gamma)J)(f(w))$ for every $\gamma \in \Gamma$ and $\nu_P$-almost every $w \in G/P$, we conclude that the unique essential value of $f$ commutes with $JMJ$ and so lies in $M$. This shows that $f \in \C1 \ovt M$. Thus, we have $\cB = \C1 \ovt M$.
\end{proof}

\subsection{Proof of Corollary \ref{main rep}}
Before proving Corollary \ref{main rep}, we make the following easy observation regarding unitary representations of property (T) groups. 

\begin{lem}\label{lem:(T)}
Let $\Lambda$ be any countable infinite group with property {\em (T)}. Then for any $\Lambda$-unitary representations $\pi$ and  $\rho$, if $\pi$ is weakly mixing and if $\rho$ is weakly contained in $\pi$, then $\rho$ is also weakly mixing.
\end{lem}
\begin{proof}
By contradiction, assume that $\rho$ is not weakly mixing. Then $\rho$ contains a finite dimensional representation. Up to replacing $\rho$ by this finite dimensional subrepresentation, we may assume that $\rho$ is finite dimensional and is weakly contained in $\pi$. Then $\rho \otimes \overline \rho$ contains the trivial representation and is weakly contained in $\pi \otimes \overline \pi$. Since $\Lambda$ has property (T), it follows that $\pi \otimes \overline \pi$ contains the trivial representation and thus $\pi$ contains a finite dimensional representation. This contradicts the assumption that $\pi$ is weakly mixing.
\end{proof}
 
\begin{proof}[Proof of Corollary \ref{main rep}]
Assume that $G$ has trivial center. Let $\Gamma < G$ be any irreducible lattice. Since $\Gamma$ has property (T), $\Gamma$ has countably many finite dimensional unitary representations up to unitary conjugacy, that we denote by $\pi_n$, $n \geq 1$ (see \cite[Theorem 2.1(iv)]{Wa74}). Denote by $\pi_0 = \lambda$ the left regular representation. Theorem \ref{main peterson} implies that $\Gamma$ has countably many extreme points in the space of characters: those corresponding to the finite dimensional unitary representations $(\pi_n,H_n)$, denoted by $\tau_n$, $n \geq 1$; and the one corresponding to the left regular representation $\pi_0 = \lambda$, namely the Dirac map at the identity, denoted by $\tau_0 = \delta_e$.

Let $\pi$ be any weakly mixing $\Gamma$-unitary representation. Set $A = \rC^*_\pi(\Gamma)$ and denote by $\sigma : \Gamma \curvearrowright A$ the conjugation action defined by $\sigma_\gamma = \Ad(\pi(\gamma))$ for every $\gamma \in \Gamma$. There exists a $\mu_0$-stationary state $\phi \in \mathcal S(A)$. Since $\varphi= \phi \circ \pi$ is a $\mu_0$-character, it is in fact a genuine character by Theorem \ref{main stationary characters}. This means that $\phi$ is a tracial state. Thus, $A$ has at least one tracial state. We prove now that the left regular representation $\lambda$ is weakly contained in $\pi$ and that $A$ has a unique tracial state. 

Let $\phi \in \mathcal S(A)$ be any tracial state and denote by $\varphi = \phi \circ \pi$ the corresponding character on $\Gamma$. By the first paragraph, we may find a sequence $(\alpha_n)_{n \in \N}$ of nonnegative real numbers such that $1 = \sum_{n \in \N} \alpha_n$ and $\varphi = \sum_{n  \in \N} \alpha_n \tau_n$. 

\begin{claim}\label{claim:weak}
For every $n \in \N$ such that $\alpha_n \neq 0$, we have that $\pi_n$ is weakly contained in $\pi$.
\end{claim}

\begin{proof}[Proof of Claim \ref{claim:weak}]
We view $\pi_n$ and $\pi$ as representations of the full $\rC^*$-algebra $\rC^*(\Gamma)$. In particular $\varphi = \phi \circ \pi$ is now a state on $\rC^*(\Gamma)$. We denote by $\phi_n$ the canonical faithful normal tracial state on the finite factor $\pi_n(\Gamma)\dpr$ so that $\phi_n \circ \pi_n = \tau_n$. Note that $\phi_n$ is implemented by a cyclic vector $\xi_n \in H_n$. By uniqueness of the GNS representation, we have that $(\pi_n,H_n,\xi_n)$ is the GNS triple associated with $\tau_n$.
For all $a, b  \in \rC^*(\Gamma)$, we have 
\[\Vert \pi_n(b) \, \pi_n(a)\xi_n \Vert = \Vert ba \Vert_{2, \tau_n} \leq \frac{1}{\sqrt{\alpha_n}} \Vert ba \Vert_{2, \phi \circ \pi} \leq   \frac{1}{\sqrt{\alpha_n}}\Vert \pi(b) \Vert \cdot \Vert \pi(a) \Vert.\]
If $\pi(b) = 0$, then $\pi_n(b) = 0$. This proves our claim.
\end{proof}

Using Claim \ref{claim:weak} and Lemma \ref{lem:(T)}, we obtain that $\alpha_n = 0$ for every $n \geq 1$. Then $\varphi = \tau_0$ and Claim \ref{claim:weak} implies that $\lambda = \pi_0$ is weakly contained in $\pi$. Denote by $\Theta_{\pi, \lambda} : A \to \rC^*_\lambda(\Gamma) : \pi(\gamma) \mapsto \lambda(\gamma)$ the corresponding surjective unital $\ast$-homomorphism. Then we have $\phi = \tau_0 \circ \Theta_{\pi, \lambda}$. This shows that $A = \rC^*_\pi(\Gamma)$ has a unique tracial state.

Finally, we show that $\ker(\Theta_{\pi, \lambda})$ is the unique proper maximal ideal of $A = \rC^*_\pi(\Gamma)$. Assume that $I \subset A$ is a proper ideal and consider the quotient map $\alpha : A \to A/I$. Then the unitary representation $\rho : \Gamma \to \mathcal U(A/I) : \gamma \mapsto \alpha( \pi(\gamma))$ is weakly contained in $\pi$ and hence weakly mixing by Lemma \ref{lem:(T)}. By the first part of the proof, we know that $\lambda$ is weakly contained in $\rho$, which in turn implies that the map $\beta : A/I \to \rC^*_\lambda(\Gamma) : \alpha(\pi(\gamma)) \mapsto \lambda(\gamma)$ is a well-defined surjective unital $\ast$-homomorphism. By construction,  we have $\Theta_{\pi, \lambda} = \beta \circ \alpha$. This shows that $I = \ker(\alpha) \subset \ker(\Theta_{\pi, \lambda})$.
\end{proof}

\subsection{Proofs of Theorem \ref{letterthm:actions} and Corollary \ref{lettercor:URS}} 

\begin{proof}[Proof of Theorem \ref{letterthm:actions}]
Let $\Gamma < G$ be any irreducible lattice. Set $Z(\Gamma) = \Gamma \cap Z(G)$. Let $(X, \nu)$ be any ergodic $(\Gamma, \mu_0)$-space. For every $\gamma \in \Gamma$, set $\Fix(\gamma) = \{x \in X \mid \gamma x = x\}$. Using Theorem \ref{main NCNZ}, the following dichotomy holds.
\begin{itemize}
\item [$(\rm i)$] Either $\nu$ is $\Gamma$-invariant.
\item [$(\rm ii)$] Or there exist a proper parabolic subgroup $P \subset Q\subsetneq G$ and a $\Gamma$-equivariant measurable factor map $(X, \nu) \to (G/Q, \nu_Q)$.
\end{itemize}
Moreover, assume that the action $\Gamma \curvearrowright (X, \nu)$ is faithful and properly ergodic. In case $(\rm i)$,  Theorem \cite[Corollary 4.4]{SZ92} implies that the action $\Gamma \curvearrowright (X, \nu)$ is essentially free.

In case $(\rm ii)$, Lemma \ref{lem:free} implies that for every $\gamma \in \Gamma \setminus Z(\Gamma)$, we have $\nu(\Fix(\gamma)) = 0$. Let now $\gamma \in Z(\Gamma) \setminus \{e\}$. Then $\Fix(\gamma)$ is $\Gamma$-invariant. Since the action is faithful and ergodic, this implies that $\nu(\Fix(\gamma)) = 0$. Thus, the action $\Gamma \curvearrowright (X, \nu)$ is essentially free.
\end{proof}

\begin{proof}[Proof of Corollary \ref{lettercor:URS}]
Assume that $G$ has trivial center. Let $\Gamma < G$ be any irreducible lattice. Let $\Gamma \curvearrowright X$ be any minimal action. Choose an extreme point $\nu$ in the compact convex set of all $\mu_0$-stationary Borel probability measures on the compact metrizable space $X$.  Then the action $\Gamma \curvearrowright (X, \nu)$ is ergodic. Moreover, since the action $\Gamma \curvearrowright X$ is minimal, we have $\supp(\nu) = X$. Assume that the action $\Gamma \curvearrowright X$ is not topologically free. Then there exists $\gamma \in \Gamma \setminus \{e\}$ such that $\Fix(\gamma) = \{x \in X \mid \gamma x = x\}$ has nonempty interior. Since $\supp(\nu) = X$, this implies that $\nu(\Fix(\gamma)) > 0$ and so the action $\Gamma \curvearrowright (X, \nu)$ is not essentially free. Theorem \ref{letterthm:actions} implies that the action $\Gamma \curvearrowright (X, \nu)$ is not faithful or not properly ergodic. 

If the action is not faithful, then Margulis' normal subgroup theorem (see \cite[Theorem IV.4.10]{Ma91}) implies that $\ker(\Gamma \curvearrowright (X, \nu))$ has finite index in $\Gamma$, and we get the desired conclusion that $X$ is finite. 
If now the action is not properly ergodic, then $\nu$ is atomic. Since $\nu$ is $\mu_0$-stationary, the maximum principle implies that the set of atoms of $\nu$ is finite and $\nu$ is $\Gamma$-invariant. This implies that $\supp(\nu)$ coincides with the finite set of atoms and so $X = \supp(\nu)$ is finite in this case as well.

Let now $X \subset \Sub(\Gamma)$ be any URS. By definition of the URS and since $\Gamma \neq \{e\}$, for every $x = \Lambda \in X$, we have $\Stab_\Gamma(x) = {\rm N}_\Gamma(\Lambda) \neq \{e\}$. This implies that $\Gamma \curvearrowright X$ is not topologically free. Therefore, $X$ is finite.
\end{proof}


\appendix

\section{On the essential range of a Borel measurable function}

Let $Y$ be any Polish space and $\eta \in \Prob(Y)$ any Borel probability measure. Recall that the {\em topological support} $\supp(\eta)$ of the measure $\eta$ is the intersection of all the closed subsets $F \subset Y$ such that $\eta(F) = 1$.

\begin{defn}
Let $(X, m)$ be a standard measure space endowed with a $\sigma$-finite Borel measure and $(Y, \mathcal O)$ a topological space with a countable basis. Let $f : X \to Y$ be a Borel measurable map.
\begin{itemize}
\item The {\em essential range} or {\em essential image} $F_f$ of $f : X \to Y$ is the intersection of all the closed subsets $F \subset Y$ such that $m(X \setminus f^{-1}(F)) = 0$.
\item An element $y \in Y$ is said to be an {\em essential value} of the function $f : X \to Y$ if for every open subset $O \subset Y$ such that $y \in O$, we have $m(f^{-1}(O)) > 0$. We will denote by $V_f \subset Y$ the subset of all the essential values of $f$.
\end{itemize}
\end{defn}

\begin{lem}\label{lem:essential-range}
Let $(X, m)$ be a standard measure space endowed with a $\sigma$-finite Borel measure and $Y$ a Polish space. Let $f : X \to Y$ be a Borel measurable map. Choose a Borel probability measure $\mu \in \Prob(X)$ in the same class as $m$. Then we have
$$F_f = V_f = \supp(f_\ast \mu) \quad \text{and} \quad \mu(f^{-1}(Y \setminus F_f)) = 0.$$
\end{lem}

\begin{proof}
By definition, we have $F_f = \supp(f_\ast \mu)$. It remains to prove that $F_f = V_f$. 

We first show that $F_f \subset V_f$. Observe that $V_f \subset Y$ is a closed subset. Indeed let  $y \in \overline{V_f}$ and $O \subset Y$ be an open subset with $y \in O$. There exists $z \in V_f \cap O$, so we have $\mu(f^{-1}(O)) > 0$ and hence $y \in V_f$. We next show that $\mu(f^{-1}(V_f)) = 1$. Indeed for all $y \in Y \setminus V_f$, there exists an open subset $O_y \subset Y \setminus V_f$ such that $y \in O_y$ and $\mu(f^{-1}(O_y)) = 0$. We have $Y \setminus V_f = \bigcup_{y \in Y \setminus V_f} O_y$. Since the topology $\mathcal O$ on $Y$ has a countable basis, it follows that $\mu(Y \setminus V_f) = 0$. Thus, $V_f \subset Y$ is a closed subset satisfying $\mu(f^{-1}(V_f)) = 1$. It follows that $F_f \subset V_f$.

We next show that $V_f \subset F_f$. Let $F \subset Y$ be any closed subset satisfying $\mu(f^{-1}(F)) = 1$. Let $y \in Y \setminus F$. Since $Y \setminus F \subset Y$ is an open subset and since $\mu(f^{-1}(Y \setminus F)) = 0$, we have $y \in Y \setminus V_f$. Therefore $Y \setminus F \subset Y \setminus V_f$ and hence $V_f \subset F$. Since this holds for all closed subsets $F \subset Y$ satisfying $\mu(f^{-1}(F)) = 1$ we have $V_f \subset F_f$.
\end{proof}

Observe that if $f , g : X \to Y$ are two Borel measurable maps that agree $\mu$-almost everywhere, we have $F_f = F_g$.

\bibliographystyle{plain}

\end{document}